\documentclass[12pt]{article}
\usepackage{geometry,amsmath,amssymb,graphicx,amscd,amsthm,url}
\newcommand{\dueto}[1]{\textup{\textbf{(#1) }}}
\newcommand{\emdash}{---}

\newcommand{\nin}{\not\in}
\newcommand{\tmem}[1]{{\em #1\/}}
\newcommand{\tmop}[1]{\ensuremath{\operatorname{#1}}}
\newcommand{\tmscript}[1]{\text{\scriptsize{$#1$}}}
\newcommand{\tmtextit}[1]{{\itshape{#1}}}
\newcommand{\tmtexttt}[1]{{\ttfamily{#1}}}
\newtheorem{theorem}{Theorem}
\newtheorem{lemma}[theorem]{Lemma}
\newtheorem{proposition}[theorem]{Proposition}

\begin{document}

\title{Whittaker Functions and Demazure Operators}
\author{Ben Brubaker, Daniel Bump, and
Anthony Licata}
\maketitle

\begin{abstract}
We consider a natural basis of the Iwahori fixed vectors in the
Whittaker model of an unramified principal series representation of
a split semisimple $p$-adic group, indexed by the Weyl group. We 
show that the elements of this basis may be computed from one another by
applying Demazure-Lusztig operators. The precise identities involve correction
terms, which may be calculated by a combinatorial algorithm that is identical
to the computation of the fibers of the Bott-Samelson resolution of a Schubert
variety. The Demazure-Lusztig operators satisfy the braid and quadratic
relations satisfied by the ordinary Hecke operators, and this leads to an
action of the affine Hecke algebra on functions on the maximal torus of the
L-group.  This action was previously described by Lusztig using
equivariant K-theory of the flag variety, leading to the proof of the
Deligne-Langlands conjecture by Kazhdan and Lusztig. In the present paper, the
action is applied to give a simple formula for the basis vectors of the
Iwahori Whittaker functions.
\end{abstract}

\medbreak\noindent
It is well known that there are relations between the representation theory of
$p$-adic groups and the topology of flag varieties. More precisely, let $G$ be
a split semisimple group over a nonarchimedean local field $F$, and let
$\hat{G} ( \mathbb{C})$ be the (connected) Langlands dual group. Then the
representation theory of $G (F)$ is closely related to the topology of the
complex flag variety $X$ of $\hat{G} ( \mathbb{C})$. For example, the same
affine Hecke algebra appears as both contexts. On the one hand, Iwahori and
Matsumoto~{\cite{IwahoriMatsumoto}}, in introducing this important ring,
showed that it is a convolution ring of functions on $G (F)$ acting on the
Iwahori fixed vectors of any representation. On the other hand,
Lusztig~{\cite{LusztigK1}} interpreted it as a ring of endomorphisms of the
equivariant K-theory of the flag variety of $\hat{G} (F)$.

With this in mind, we study Whittaker functions for unramified principal
series representations of $G (F)$. Our investigations were motivated by
earlier work of Reeder~{\cite{ReederVector}} and we begin with a brief summary
of his results. Further discussion of this may be found in
Section~\ref{crresults}.

Let $\tau$ be an unramified character of the split maximal torus $T (F)$ where
$B = TN$ is the standard Borel subgroup. Then we may form the principal series
representation $M (\tau) = \tmop{Ind}_B^G (\tau)$. A Whittaker model for $M
(\tau)$ is an intertwining map $\mathcal{W}_{\tau} : M (\tau) \rightarrow
\tmop{Ind}_N^G (\psi)$ where $\psi$ is in an open $T$-orbit of the space of
characters of $N$; such an intertwiner is unique up to constant. An important
problem is the characterization of functions $\mathcal{W}_{\tau} (\phi) : G
\rightarrow \mathbb{C}$ for distinguished vectors $\phi \in M (\tau)$.

Let $J$ be the Iwahori subgroup of $G (F)$. The space $M (\tau)^J$ of
$J$-fixed vectors has dimension equal to the order of the Weyl group $W$ of
$G$. We will be primarily concerned with two particular bases of $M (\tau)^J$.
First, we have the ``standard basis'' $\{\Phi_w^{\tau} \; |
\; w \in W\}$ whose elements are defined by
\begin{equation}
\label{phidef}\Phi_w^{\tau} (buk) = \left\{ \begin{array}{ll}
     \delta^{1 / 2} \tau (b) & \text{if $u = w$,}\\
     0 & \tmop{otherwise},
   \end{array} \right. \hspace{1em} (b \in B (F), u \in W \text{and } k \in J)
\end{equation}
where $\delta$ is the modular character of $B (F)$. These functions are well-defined
according to the decomposition $G = \coprod_{w \in W} BwJ$. Our formulas will
be simpler if we use the basis obtained by summing $\Phi_w$'s according to the
Bruhat order:
\begin{equation}
\label{phitildef}
\tilde{\Phi}_w^{\tau} = \sum_{u \geqslant w} \Phi^{\tau}_u,
\end{equation}
so that if $u, w \in W$
\[ \tilde{\Phi}_w^{\tau} (buk) = \left\{ \begin{array}{ll}
     \delta^{1 / 2} \tau (b) & \text{if $u \geqslant w$,}\\
     0 & \tmop{otherwise} .
   \end{array} \right. \]
In particular $\tilde{\Phi}^{\tau}_1$ is the standard spherical vector:
\[ \tilde{\Phi}^{\tau}_1 (bk) = \delta^{1 / 2} \tau (b) \hspace{2em} b \in B,
   k \in K. \]
In {\cite{ReederVector}} Reeder investigated the functions $\mathcal{W}_{\tau}
( \tilde{\Phi}_w^{\tau})$ evaluated on the maximal torus, using methods of
Casselman and Shalika. The key idea is that one can compute the effect of
intertwining operators on Iwahori fixed vectors before or after applying the
Whittaker functional.

Let $\Lambda$ be the weight lattice of $\hat{G}$, which is the group $X^{\ast}
( \hat{T})$ of rational characters of the maximal torus $\hat{T}$ of $\hat{G}$
that is dual to $T$.  Thus $\Lambda$ may be identified with the group
$X_{\ast} (T)$ of rational one parameter subgroups of $T$, isomorphic to $T
(F) / T ( \mathfrak{o})$, where $\mathfrak{o}$ is the ring of integers in
$F$. If $\lambda$ is a dominant weight of $\hat{T}$, let $a_{\lambda} \in T
(F)$ be a representative of the corresponding coset in $T (F) / T
(\mathfrak{o})$.  Casselman~{\cite{Casselman}} describes, in addition to the
basis $\Phi_w$ of Iwahori fixed vectors, a more subtly defined basis which we
refer to as the \tmtextit{Casselman basis} $\{f_w^{\tau} \}$. Reeder gives a
simple formula for $\mathcal{W}_{\tau} (f_w^{\tau}) (a_{\lambda})$, but a
similar closed formula for $\mathcal{W}_{\tau} ( \tilde{\Phi}_w^{\tau})
(a_{\lambda})$ is more difficult. However he found a recursive algorithm for
the change of basis between the Casselman basis and the basis $\Phi_w$. This
algorithm, which we implemented in \tmtexttt{Sage}, allowed us to compute the
Whittaker functions evaluated at any fixed $a_{\lambda}$ and these
calculations were an important tool in our investigation.

Furthermore, Reeder saw that the $\mathcal{W}_{\tau} ( \tilde{\Phi}_w^{\tau})
(a_{\lambda})$ should be related to the coherent cohomology of line bundles
over the flag varieties and their Schubert varieties, and gave such
interpretations for particular Weyl group elements{\emdash}those corresponding
to the long element of a Levi subgroup of $W$.

We will exhibit connections between Iwahori Whittaker functions and the geometry
of Schubert varieties that hold for all Weyl group elements. Our
starting point is a recursive relation for the Whittaker function of
$\tilde{\Phi}_w$ in terms of Bruhat order (Theorem~\ref{magictheorem}). It is
obtained by considering an operator on $M (\tau)$ which acts as an idempotent
in the isomorphism with the finite Hecke algebra. From this we prove that the
Iwahori Whittaker functions become Demazure characters when $q^{- 1}$ is
specialized to~$0$, where $q$ is the cardinality of the residue field. See
Theorem~\ref{wasthmone} for a precise statement.

Instead of the Whittaker model, one may consider the Iwahori fixed vectors in
the spherical model of the representation. This has been investigated by
Ion~\cite{Ion}, who also finds that Demazure operators play a role, and
(generalizing the Macdonald formula for the spherical function) the functions
on the $p$-adic group may be expressed in terms of the nonsymmetric Macdonald
polynomials.

We briefly recall the definition of Demazure characters and their relation to
cohomology of Schubert varieties. Given $w \in W$, let $X_w$ be the
corresponding Schubert cell in $X = \hat{G} ( \mathbb{C}) / \hat{B} (
\mathbb{C})$, where $\hat{B}$ is the standard Borel subgroup. Thus $X_w$ is
the closure of the open Schubert cell $Y_w$, which is the image in $X$ of
$\hat{B} (\mathbb{C}) w \hat{B} (\mathbb{C})$. If $\lambda$ is
a dominant weight of $\hat{G}$, then $\lambda$ determines a line bundle
$\mathcal{L}_{\lambda}$ on $X$. The space $H^0 (X_w, \mathcal{L}_{\lambda})$
of sections is a module for the standard maximal torus $\hat{T} ( \mathbb{C})$
of $\hat{G} ( \mathbb{C})$, and the Demazure character formula computes its
character.

To describe the Demazure character formula, let $\alpha_1, \cdots, \alpha_r$
be the simple roots of $\hat{G}$ and let $s_1, \cdots, s_r$ be the
corresponding simple reflections. The Demazure operators are defined on the
ring $\mathcal{O}( \hat{T})$ of rational functions on $\hat{T}$ by
\begin{equation}
\label{demdef}
\partial_{\alpha_i} f ( \boldsymbol{z}) = \partial_i f ( \boldsymbol{z}) = \frac{f
    ( \boldsymbol{z}) - \boldsymbol{z}^{- \alpha_i} f (s_i \boldsymbol{z})}{1 -
   \boldsymbol{z}^{- \alpha_i}},
\end{equation}
for $\boldsymbol{z} \in \hat{T} ( \mathbb{C})$. The operators $\partial_i$ are
idempotent and satisfy $s_i \circ \partial_i = \partial_i$. They also satisfy
the braid relations for the Weyl group. This implies that if $\mathfrak{w} =
(s_{h_1}, \cdots, s_{h_d})$ is a reduced word for $w$, so that $w = s_{h_1}
\cdots s_{h_d}$ is a reduced decomposition of $w$ into a product of simple
reflections then we may define $\partial_w : = \partial_{h_1} \cdots
\partial_{h_d}$ and this is well-defined. In particular, if $w_0$ is the long
element and $\lambda$ is a dominant weight, then $\partial_{w_0}
\boldsymbol{z}^{\lambda}$ is the character of the irreducible representation with
highest weight $\lambda$. The Demazure character formula asserts that for an
arbitrary Weyl group element $w$ the trace of $\boldsymbol{z} \in \hat{T} (
\mathbb{C})$ on $H^0 (X_w, \mathcal{L}_{\lambda})$ is $\partial_w
\boldsymbol{z}^{\lambda}$.

A key ingredient of the proof of the Demazure character formula is the
Bott-Samelson resolution of the (possibly singular) Schubert variety $X_w$.
Depending on the reduced word $\mathfrak{w}$ we may construct the nonsingular
\tmtextit{Bott-Samelson variety} $Z_{\mathfrak{w}}$. See
Bott-Samelson~{\cite{BottSamelson}}, Demazure~{\cite{DemazureDesing}},
Andersen~{\cite{Andersen}} and Kumar~{\cite{Kumar}}. There is then a
birational morphism $Z_{\mathfrak{w}} \longrightarrow X_w$. Since the map is
birational, the fiber over a generic point is just is a single point. However
even if $X_w$ is nonsingular, this map may have nontrivial fibers over some
points. Pulling the line bundle back to $Z_{\mathfrak{w}}$ does not change its
space of sections, and over the Bott-Samelson variety, the cohomology of
$\mathcal{L}_{\lambda}$ may be computed inductively using the Leray spectral
sequence, leading to the Demazure character formula.

On the other hand, returning to Whittaker functions over a nonarchimedean
local field $F$, $\boldsymbol{z} \in \hat{T} (\mathbb{C})$ parametrizes an
unramified character $\tau = \tau_{\boldsymbol{z}}$ of $T (F)$, which may be
parabolically induced to $G (F)$. The set $\{ \mathcal{W}_{\tau}
\tilde{\Phi}^{\tau}_w \}$ indexed by elements $w$ of the Weyl group
gives a natural basis of the space of Iwahori fixed vectors in the
Whittaker model. 
If $\lambda$ is a dominant weight of $\hat{T}$, then $\lambda$ parametrizes a
coset in $T (F) / T (\mathfrak{o})$, where $\mathfrak{o}$ is the ring of
integers in $F$. Let $a_{\lambda} \in T (F)$ be a representative. Then
\[ \mathcal{W}_{\tau} \tilde{\Phi}^{\tau}_w (a_{\lambda}) = \delta^{1 / 2} (a_{\lambda})
   P_{\lambda, w} (\boldsymbol{z}, q^{- 1}) \]
where $\delta$ is the modular character of $B (F)$, and $P_{\lambda, w}$ is a
rational function in $\boldsymbol{z}$ -- that is, a finite linear combination of
weights -- whose coefficients are polynomials in $q^{- 1}$. The factor
$\delta^{1 / 2} (a_{\lambda})$ is a constant, independent of $\boldsymbol{z}$,
and removing it is harmless. For example, multiplying by $\delta^{1 / 2}
(a_{\lambda})$ commutes with the Demazure operators.

We will use Bott-Samelson varieties to exhibit a further connection between
representations and geometry---a certain dictionary between Whittaker
functions and Schubert varieties. Our philosophy is that in the geometric
picture, the various varieties that appear, including Schubert and
Bott-Samelson varieties, can be associated with polynomials depending on the
spectral parameter $\boldsymbol{z}$, a dominant weight $\lambda$ and a parameter
$v$. Thus given a variety $V$ with an action of the Borel subgroup
of $\hat G$, together with an equivariant map $V\to X$, one may hope
to construct an invariant $\boldsymbol{V}(\boldsymbol{z},\lambda,v)$ which
is a polynomial in $v$. We will not give a systematic theory of such
invariants, but we will make such an association for the particular varieties
that come up in our study of Whittaker functions.

In order to make this dictionary more apparent, let us use the notation
\begin{equation}
\label{bigxdef}
\boldsymbol{X}_w ( \boldsymbol{z}, \lambda, v) := P_{- w_0 \lambda, ww_0} (
   \boldsymbol{z}^{- 1}, v),
\end{equation}
where $v$ is an indeterminate and $w_0$ is the long Weyl group element. Note
that if $\lambda$ is a dominant weight, then so is $- w_0 \lambda$. The
construction of $\boldsymbol{X}_w$ was by way of Whittaker functions, but we are
now thinking of it as an invariant associated to the Schubert variety $X_w$. We
will also use the notation $\boldsymbol{Y}_w$ where we set
\[ \boldsymbol{X}_w = \sum_{u \leqslant w} \boldsymbol{Y}_u, \hspace{2em} \text{and so}
\hspace{2em} \boldsymbol{Y}_w = \sum_{u \leqslant w} (- 1)^{l (w) - l (u)} \boldsymbol{X}_u. \]
Thus $\boldsymbol{Y}_w$ may be obtained from Whittaker functions using $\mathcal{W}_{\tau}
\Phi_w^{\tau}$ instead of $\mathcal{W}_{\tau} \tilde{\Phi}^{\tau}_w$.

We will see that if we replace $v$ by $0$ then $\boldsymbol{X}_w$ becomes a
Demazure character; that is,
\begin{equation}
  \label{xasdc} \boldsymbol{X}_w ( \boldsymbol{z}, \lambda, 0) = \partial_w
  \boldsymbol{z}^{\lambda} .
\end{equation}
Since the Demazure character is the coherent cohomology of a line bundle on
$X_w$, this is the first evidence of a connection between Whittaker functions
and the geometry of the Schubert varieties.

In order to describe a deeper connection, it is necessary to understand the
complete polynomial $P_{\lambda, w}$, not just its constant term, in relation
to Schubert varieties. Corresponding to the Bott-Samelson variety
$Z_{\mathfrak{w}}$, define
\begin{equation}
  \label{bfzdef} \boldsymbol{Z}_{\mathfrak{w}} ( \boldsymbol{z}, \lambda, v) =
  \mathfrak{D}_{h_1} \cdots \mathfrak{D}_{h_d} \boldsymbol{z}^{\lambda}
\end{equation}
where
\begin{equation}
  \label{frakddef}  \mathfrak{D}_i = (1 - v \boldsymbol{z}^{- \alpha_i})
  \partial_i .
\end{equation}
The Bott-Samelson variety may be built up from a point by successive fiberings
by $\mathbb{P}^1$, and we think of the the application of the operators
$\mathfrak{D}_i$ as an algebraic analog of this process.

Our thesis is that the relationship between $\boldsymbol{X}_w$, which is
essentially the Whittaker function, and $\boldsymbol{Z}_{\mathfrak{w}}$ is
identical to the relationship between the Schubert varieties $X_w$ and the
Bott-Samelson varieties $Z_{\mathfrak{w}}$. The $\boldsymbol{Y}_w$ then
correspond to the open Schubert varieties $Y_w.$

To explain further, we may write
\[ \boldsymbol{X}_w =\boldsymbol{Z}_{\mathfrak{w}} - v \times \text{``correction
   terms''} . \]
The correction terms are linear combinations of $\boldsymbol{Z}_{\mathfrak{u}}$
for a few particular reduced words $\mathfrak{u}$ of $u \leqslant w$ in the
Bruhat order. The determination of the correction terms is a combinatorial
matter, and the point is that the combinatorics are identical between the
Whittaker ($\boldsymbol{X}_w$ and $\boldsymbol{Z}_{\mathfrak{w}}$) and the
Schubert ($X_w$ and $Z_{\mathfrak{w}}$) relationships.

In order to explain the relationship most simply, it is useful to introduce
\tmtextit{partial Bott-Samelson varieties} $Z_{s, w}$, where $s = s_{\alpha}$
is a simple reflection and $w$ is a Weyl group element such that $sw > w$.
This is a fiber bundle over $\mathbb{P}^1$ in which the fibers are $X_w$.
Furthermore, it is equipped with a birational morphism $Z_{s, w} \to X_{sw}$.
In keeping with the algebraic analogy described above, we define
\[ \boldsymbol{Z}_{s, w} = \mathfrak{D}_i \boldsymbol{X}_w, \hspace{1em}
   \text{where} \hspace{1em} \alpha = \alpha_i . \]
Just as in~(\ref{bfzdef}), the operator $\mathfrak{D}_i$
should be thought of as an algebraic analog of the $\mathbb{P}^1$.

We find that
\[ \boldsymbol{X}_{sw} =\boldsymbol{Z}_{s, w} - v \times \text{``correction
   terms''.} \]
Here the correction terms have the following combinatorial description. Let $H
(w, s)$ be the set of $u \in W$ such that both $u, su \leq w$ in the Bruhat order. 
If $u \in H (w,s)$ and $t \leqslant u$ then $t \in H (w, s)$. The set $H (w, s)$ may be empty
or it may have a few maximal elements $u_1, \cdots, u_k$. If it is empty, then
$\boldsymbol{X}_{sw} =\boldsymbol{Z}_{s, w}$, while if $H (w, s)$ has a unique
maximal element $u_1$, then
\[ \boldsymbol{X}_{sw} =\boldsymbol{Z}_{s, w} - v\boldsymbol{X}_{u_1} . \]
If $H (w, s)$ has two maximal elements $u_1$ and $u_2$, and $\{x \leqslant
u_1, u_2 \}$ has a unique maximal element $u_3$, then
\begin{equation}
  \label{correxam} \boldsymbol{X}_{sw} =\boldsymbol{Z}_{s, w} -
  v\boldsymbol{X}_{u_1} - v\boldsymbol{X}_{u_2} + v\boldsymbol{X}_{u_3} .
\end{equation}
This latter case occurs, for example, in
Type $A_3$ when $s = s_1$ and $w = s_2 s_1 s_3 s_2$, with $u_1 = s_1 s_2 s_1$,
$u_2 = s_1 s_3 s_2$ and $u_3 = s_1 s_2$.
This illustrates the combinatorial description of the Whittaker functions in
terms of the operators $\mathfrak{D}_i$.

Now let us consider an analogous geometric problem leading to the
\tmtextit{same combinatorial decomposition} as in (\ref{correxam}). We recall
that there is a birational morphism $\mu : Z_{s, w} \longrightarrow
X_{sw}$. The fiber over every point $x \in X_{sw}$ is either a point or a
$\mathbb{P}^1$. Let us determine the subvariety $E$ of $X_{sw}$ where the
fiber is $\mathbb{P}^1$. The combinatorics of this question are again
controlled by the same set $H (w, s)$. Indeed, $E$ is just the union of the
$X_u$ as $u$ runs through the maximal elements of $H (w, s)$. For
computational purposes, however, it may be useful to describe $E$ by recording
the overlapping of its irreducible components. Thus in the case where $H (w,
s)$ has two maximal elements $u_1$ and $u_2$, and $\{x \leqslant u_1, u_2 \}$
has a unique maximal element $u_3$, $E$ has two irreducible components
$X_{u_1}$ and $X_{u_2}$, whose intersection is $X_{u_3}$. So we write formally
\begin{equation}
  E = X_{u_1} + X_{u_2} - X_{u_3} . \label{suggestive}
\end{equation}
The three Weyl group elements that appear with nonzero coefficient are the
same in (\ref{correxam}) and (\ref{suggestive}), and we will prove that this
is always true.

The notation (\ref{suggestive}) is intended to be suggestive, but we can make
it precise as follows. If $S$ is a subset of $X$ we will write symbolically
\begin{equation}
  \label{suggested} S = \sum_{w \in W} c (w) X_w, \hspace{2em} c (w) \in
  \mathbb{Z},
\end{equation}
to mean that if $y \in X$ then
\[ \sum_{\tmscript{\begin{array}{c}
     w \in W\\
     y \in X_w
   \end{array}}} c (w) = \left\{ \begin{array}{ll}
     1 & \text{if $y \in S$}\\
     0 & \text{otherwise.}
   \end{array} \right. \]
We may now state a theorem connecting the correction terms in the expansions
of the Whittaker functions and the fibers of the maps $Z_{s, w}
\longrightarrow X_{sw}$.

\begin{theorem}
  \label{thmain}Let $s$ be a simple reflection, and let $w \in W$ such that
  $sw > w$. Then there exist integer coefficients $c_{s, w} (u) = c (u)$
  defined for $u \in W$ such that
  \[ \boldsymbol{X}_{sw} =\boldsymbol{Z}_{s, w} - v \sum_u c (u)\boldsymbol{X}_u .
  \]
  Moreover, there is a corresponding geometric identity
  \[ S = \sum_{u \in W} c (u) X_u \]
  where $S$ is the subset of $X_{sw}$ such that the fiber of the map $Z_{s, w}
  \longrightarrow X_{sw}$ over $y \in X_{sw}$ is a point if $y \nin S$, or
  $\mathbb{P}^1$ if $y \in S$. The coefficient $c (u)$ is zero unless both $u,
  su < w$.
\end{theorem}

In the special case $v = 0$, this immediately implies the specialization to
Demazure characters in (\ref{xasdc}). When $v = 1$ they also
have a simple specialization. Let $\rho$ be the Weyl vector for $\hat{G} (
\mathbb{C})$, that is, half the sum of the positive roots.

\begin{theorem}
  \label{qone}Given a dominant weight $\lambda$, for any $w \in W$:
  \begin{equation}
    \label{permutahedron} \boldsymbol{X}_w ( \boldsymbol{z}, \lambda, 1) = \sum_{u
    \leqslant w} (- 1)^{l (u)} \boldsymbol{z}^{u (\rho + \lambda) - \rho} .
  \end{equation}
\end{theorem}

As alluded to earlier, the operators $\mathfrak{D}_i$ have an interpretation in terms of Hecke
algebras. The Hecke algebra $\mathcal{H}_v$ associated with the Weyl
group $W$ has generators $T_i$, one for each Weyl group element $s_i$. They
satisfy the same braid relations as the $s_i$, together with the quadratic
relations $T_i^2 = (v - 1) T_i + v$. We will show that the operators
$\mathfrak{D}_i - 1$ satisfy these relations. Therefore we obtain a
representation of $\mathcal{H}_v$ on $\mathbb{C}[v,v^{-1}]\otimes\mathcal{O}(
\hat{T})$ in which $T_i$ acts as the operator $\mathfrak{D}_i - 1$. The
operators $\mathfrak{D}_i - 1$ are essentially the same as the
\tmem{Demazure-Lusztig operators} which first appeared in
Lusztig~\cite{LusztigK1}, equation~(8.1). See equation (\ref{usvsthem}) below.

Once this representation of $\mathcal{H}_v$ is constructed, the
functions $\boldsymbol{X}_w$ or $\boldsymbol{Y}_w$ have the following explicit
description. Denote the effect of $\phi \in \mathcal{H}_v$ on $f \in
\mathbb{C}[v,v^{-1}]\otimes\mathcal{O}( \hat{T})$ by $\phi \cdot f$. 

\begin{theorem}
\label{amazing}
Let $\lambda$ be a dominant weight. Then
\[ \boldsymbol{Y}_w (\lambda) = T_w \cdot \boldsymbol{z}^{\lambda}, \hspace{2em}
   \boldsymbol{X}_w (\lambda) = \sum_{u \leqslant w} T_u \cdot
   \boldsymbol{z}^{\lambda} . \]
\end{theorem}

\noindent
This shows that the action of the Hecke algebra by
Demazure-Lusztig operators gives a tidy formula for a basis of the Iwahori
fixed Whittaker functions.

An equivalent formulation more directly in terms of Whittaker functions is
given in Theorem~\ref{midmagic}. The relation between $\mathcal{W}_\tau$ and
$T_w$ appearing in the statement of Theorem~\ref{amazing} is more elaborate
and we supply a second proof based on machinery developed in
Section~\ref{HeckeAlgebras}. This machinery is also used to obtain the
extension of the action to the affine Hecke algebra, a point we discuss next.

The affine Iwahori Hecke algebra, denoted $\tilde{\mathcal{H}}_v$, is generated
by $\mathcal{H}_v$ together an abelian subgroup $\zeta^{\Lambda}$ isomorphic
with the weight lattice $\Lambda$ of $\hat{T}$. We will show
in Theorem~\ref{cherednik} that the above
representation of $\mathcal{H}_v$ may be extended to a representation of
$\tilde{\mathcal{H}}_v$ in which $\zeta^{\Lambda}$ acts by translation: if
$\lambda \in \Lambda$ then $\zeta^{\lambda} \in \zeta^{\Lambda}$ is an element
that acts on functions of $\mathbb{C}[v,v^{-1}]\otimes\mathcal{O}( \hat{T})$ by multiplication by
$\boldsymbol{z}^{- \lambda}$. The resulting representation of
$\tilde{\mathcal{H}}_v$ was previously considered by
Lusztig~\cite{LusztigK1}. Indeed
$\mathbb{C}[v,v^{-1}]\otimes\mathcal{O}( \hat{T})$ is isomorphic to
the complexified equivariant K-group 
$\mathbb{C}\otimes_{\mathbb{Z}}K_M(X)$ of the flag variety $X$, where $M=\hat{G}\times GL_1$,
and Lusztig constructions an action of the Hecke algebra (over $\mathbb{Z}$)
on $K_M(X)$. The
same representation also appeared in Cherednik~{\cite{Cherednik}} in the
context of double affine Hecke algebras. In our context, it appears naturally
in the theory of Whittaker functions. This representation is irreducible, and
may be described as follows. The element $\boldsymbol{z}^{- \rho}$ of
$\mathbb{C}[v,v^{-1}]\otimes\mathcal{O}( \hat{T})$, where $\rho$ is the Weyl vector
(half the sum of the positive roots) for $\hat{G}$, is annihilated by the
$\mathfrak{D}_i$, so $\mathcal{H}_v$ acts on the one-dimensional span of this
vector by the sign character that sends $T_i$ to $- 1$. The representation of
$\tilde{\mathcal{H}}_v$ on $\mathcal{O}( \hat{T})$ is the irreducible
representation induced by this one-dimensional representation of
$\mathcal{H}_v$.

Let $\widetilde{\mathcal{H}}$ be the specialization of $\widetilde{\mathcal{H}}_v$ to
a complex algebra obtained by putting $v=q^{-1}$.
It was shown by Iwahori and Matsumoto~{\cite{IwahoriMatsumoto}}
that $\tilde{\mathcal{H}}$ is the convolution ring of compactly supported
$J$-biinvariant functions on $G (F)$. It therefore acts on the
finite-dimensional vector space space $M (\tau)^J$ of Iwahori invariants by
convolution on the right. This action depends on $\tau$, and the isomorphism
class of $M (\tau)$ is determined by the isomorphism class of this
representation of $\tilde{\mathcal{H}}$. If $\boldsymbol{z} \in \hat{T}
(\mathbb{C})$ is in general position then $M (\tau_{\boldsymbol{z}})$ is
irreducible, and in this case, the isomorphism class of $M
(\tau_{\boldsymbol{z}})$ and $M (\tau_{\boldsymbol{z}'})$ are the same if and only
if $\boldsymbol{z}, \boldsymbol{z}' \in \hat{T} (\mathbb{C})$ are in the same
$W$-orbit.

On the other hand, we have seen that there is a representation of
$\tilde{\mathcal{H}}$ on $\mathbb{C}[q,q^{-1}]\otimes\mathcal{O}(\hat{T})$ 
which is constructed in this paper using the theory of Whittaker
functions, and which was constructed earlier by Lusztig in
connection with equivariant K-theory. This representation was
applied by Lusztig in~\cite{LusztigK1} and in Kazhdan
and Lusztig~\cite{KL2,KL3} to construct all representations having
an Iwahori fixed vector. To see that one may construct at least
the principal series representations, fix $\boldsymbol{z}_0\in\hat{T}(\mathbb{C})$.
Consider the ideal $\mathcal{J}_{\boldsymbol{z}_0}$ of functions in
$\mathcal{O}(\hat{T})$ that vanish on the $W$-orbit of $\boldsymbol{z}_0$.
It is clear from the definitions that this ideal is closed under the
operators $\mathfrak{D}_i$, as well as $\zeta^\Lambda$,
and so there is an induced action of $\tilde{\mathcal{H}}$ on
$\mathcal{O}(\hat{T})/\mathcal{J}_{\boldsymbol{z}_0}$. By
Lemma~7 of Bernstein and Rumelhart~\cite{Bernstein} every irreducible
constitutent of this representation is the space of Iwahori fixed vectors in
an irreducible representation of $G(F)$.

We thank Gautam Chinta, Paul Gunnells and Anne Schilling for their interest in
the project and helpful conversations during the writing of this
paper. This work was supported by NSF grants DMS-0652817, DMS-0844185, and
DMS-1001079.

\section{Preliminaries}\label{crresults}

Let $G$ be a split semisimple Chevalley group. By this we mean an affine
algebraic group scheme over $\mathbb{Z}$, whose Lie algebra
$\mathfrak{g}_{\mathbb{Z}}$ has a fixed Chevalley basis defined over
$\mathbb{Z}$ corresponding to a root system $\Delta$. The elements of the
Chevalley basis are the nilpotent elements $X_{\alpha}$ where $\alpha$ runs
through all roots and the coroots $H_{\alpha} = [X_{\alpha}, X_{- \alpha}]$
where $\alpha$ runs through the simple positive roots. The structure constants
of the Chevalley basis are all integers (cf.~{\cite{Chevalley}}). If $\alpha$
is a root, let $x_{\alpha} : \mathbb{G}_{\text{a}} \longrightarrow G$ be the
one parameter subgroup tangent to $X_{\alpha}$. Let $T$ be the split maximal
torus whose Lie algebra is spanned by the $H_{\alpha}$, and let $N$
(resp.~$N_-$) be the unipotent subgroup whose Lie algebra is spanned by the
$X_{\alpha}$ (resp.~$X_{- \alpha}$) as $\alpha$ runs through the positive
roots. Then $B = TN$ is the standard Borel subgroup of $G$.

Let $F$ be a nonarchimedean local field and $\mathfrak{o}$ its ring of
integers, $\mathfrak{p}$ the maximal ideal of $\mathfrak{o}$, and let $q = |
\mathfrak{o} / \mathfrak{p} |$. We will denote the residue field
$\mathbb{F}_q$. The group $G (F)$ has $K = G ( \mathfrak{o})$ as a maximal
compact subgroup. The group $X_{\ast} (T)$ of rational cocharacters is
isomorphic to $T (F) / T ( \mathfrak{o})$, where the one-parameter subgroup
$\varphi \in X_{\ast} (T)$ corresponds to the coset $\varphi (\varpi) T (
\mathfrak{o})$ with $\varpi$ a prime element.

If $w$ is an element of the Weyl group $W$, we will choose a fixed
representative of $w$ in $G ( \mathfrak{o})$, and by abuse of notation we will
denote this element also as $w$. Nothing will depend on this choice
in any essential way. We will denote by $w_0$ the long Weyl group element.

A character $\tau$ of $T (F)$ is called
{\tmem{unramified}} if it is trivial on $T ( \mathfrak{o})$. We will let $W$
act on the right on unramified characters $\tau$, so that
\[ (\tau w) (t) = \tau (wtw^{- 1}), \hspace{2em} w \in W, \hspace{2em} t \in T
   (F) . \]
Since $\tau$ is unramified, this does not depend on the choice of
representative $w$ of the Weyl group element.

Let $\hat{G}$ be the connected Langlands dual group. It is an algebraic group
over $\mathbb{C}$ whose root data are dual to $G$. Thus if $\Delta$ is the
root system of $G$ with respect to $T$, we will denote by $\alpha
\longrightarrow \alpha^{\vee}$ the bijection of $\Delta$ with system
$\Delta^{\vee}$ of coroots, and $\Delta^{\vee}$ may be regarded as the root
system of $\hat{G}$. If $\hat{T}$ is a maximal torus of $\hat{G}$ then the
group $X_{\ast} (T)$ of rational one-parameter subgroups of $T$ is identified
with the group $X^{\ast} ( \hat{T})$ of rational characters. Thus we have a
homomorphism
\begin{equation}
  \label{canmap} T (F) \longrightarrow T (F) / T ( \mathfrak{o}) \cong
  X_{\ast} (T) \cong X^{\ast} ( \hat{T}) .
\end{equation}
If $\boldsymbol{z} \in \hat{T} ( \mathbb{C})$ and $\lambda \in X^{\ast} (
\hat{T})$ we will denote by $\boldsymbol{z}^{\lambda}$ the application of the
character $\lambda$ to $\boldsymbol{z}$. Also if $t \in T (F)$ we may apply the
homomorphism (\ref{canmap}) to $t$ and apply the resulting rational character
of $\hat{T}$ to $\boldsymbol{z}$; we will denote the result by $\tau_{\boldsymbol{z}}
(t)$. Thus $\tau_{\boldsymbol{z}}$ is an unramified character of $T$ and
$\boldsymbol{z} \longmapsto \tau_{\boldsymbol{z}}$ is an isomorphism of $\hat{T} (
\mathbb{C})$ with the group of unramified characters of $T (F)$.

If $\lambda \in X^{\ast} ( \hat{T})$, we will denote by $a_{\lambda}$ a
representative of the coset in $T (F) / T ( \mathfrak{o})$ corresponding to
$\lambda$ by the isomorphism in (\ref{canmap}).

Let $\tau = \tau_{\boldsymbol{z}}$ be such an unramified character. Let $M (\tau)$
be the space of the representation of $G (F)$ induced from $\tau$. This is the
space of locally constant functions $f : G (F) \longrightarrow \mathbb{C}$
that satisfy
\[ f (bg) = (\delta^{1 / 2} \tau) (b) f (g), \hspace{2em} b \in B (F), \]
where $\delta : B (F) \longrightarrow \mathbb{R}^{\times}$ is the modular
character. The standard intertwining integral $\mathcal{A}_w : M (\tau)
\longrightarrow M (\tau w)$ is
\[ \mathcal{A}_w^{\tau} \Phi (g) = \int_{N (F) \cap w^{- 1} N_- (F) w} \Phi
   (wng) \hspace{0.25em} dn. \]
The integral is convergent for $\tau = \tau_{\boldsymbol{z}}$ with $|
\alpha^{\vee} ( \boldsymbol{z}) | < 1$ when $\alpha \in \Delta^+$. It makes sense
for other $\boldsymbol{z}$ by meromorphic continuation in a suitable sense.

In order to convert between this notation (which is the same as Reeder
{\cite{ReederVector}}) and that of Casselman~{\cite{Casselman}}, bear in mind
that $\mathcal{A}^{\tau}_w$ is Casselman's $T_{w^{- 1}}$. Due to this
difference, the Weyl group action on characters is a right action: $w : \tau
\longrightarrow \tau w$.

Let $J$ be the Iwahori subgroup of $G (F)$. It is the inverse image of $B (
\mathbb{F}_q)$ under the natural map $K \longrightarrow G ( \mathbb{F}_q)$.
The dimension of the space $M (\tau)^J$ of $J$-fixed vectors is $|W|$. Let
$\{\Phi_w^\tau\}$ and $\{\widetilde{\Phi}_w^\tau\}$ be the bases of $M (\tau)^J$
defined by (\ref{phidef}) and (\ref{phitildef}). The
basis element $\Phi_w^\tau$ is denoted $\phi_w$ by Casselman~{\cite{Casselman}}.

\begin{lemma}
  \label{wellknow}If $n \in N (F)$ and $w_0 n \in Bw_0 J$ then $n \in N (
  \mathfrak{o})$.
\end{lemma}

\begin{proof}
  Using the Iwahori factorization $J = N_- ( \mathfrak{p}) T ( \mathfrak{o}) N
  ( \mathfrak{o})$ we may write $w_0 n = bw_0 \gamma n_1$ with $b \in B (F)$,
  $\gamma \in N_- ( \mathfrak{p}) T ( \mathfrak{o})$ and $n_1 \in N (
  \mathfrak{o})$. Since $w_0 \gamma w_0^{- 1} \in B (F)$, this implies that
  $nn_1^{- 1} \in N (F) \cap w_0^{- 1} B (F) w_0$, so $n = n_1 \in N (
  \mathfrak{o})$.
\end{proof}

\begin{lemma}
  \label{notdomzer}Let $\lambda \in X^{\ast} ( \hat{T})$. Then $\mathcal{W}_{\tau}
  \Phi^{\tau}_w (a_{\lambda}) = 0$ if $\lambda$ is not dominant.
\end{lemma}

\begin{proof}
  See Casselman and Shalika~{\cite{CasselmanShalika}} Lemma~5.1. (The proof
  there obviously applies to all Iwahori-fixed Whittaker functions since $J$
  contains $N ( \mathfrak{o})$.)
\end{proof}

\begin{proposition}
  \label{wflongev}
  Given an unramified character $\tau$ of $T(F)$ we have $\Phi^{\tau}_{w_0} =
  \tilde{\Phi}^{\tau}_{w_0}$ and
  \[ \mathcal{W}_{\tau} \Phi_{w_0}^{\tau} (a_{\lambda}) = \left\{ \begin{array}{ll}
       \delta^{1 / 2} (a_{\lambda}) \boldsymbol{z}^{w_0 \lambda} & \text{if
         $\lambda$ is dominant,}\\ 0 & \text{otherwise} .
     \end{array} \right. \]
\end{proposition}

\begin{proof}
  The fact that $\tilde{\Phi}_{w_0} = \Phi_{w_0}$ is clear since $w_0$ is
  maximal in $W$. For the last assertion by Lemma~\ref{notdomzer} we may
  assume that $\lambda$ is dominant. Denoting $a = a_{\lambda}$,
  \[ \mathcal{W}_{\tau} \Phi_{w_0}^{\tau} (a) = \int_{N (F)} \Phi_{w_0}^{\tau}
     (w_0 na) \psi (n) \hspace{0.25em} dn. \]
  We make the variable change $n \longrightarrow ana^{- 1}$ whose Jacobian is
  $\delta (a)$, and the integral becomes
  \begin{eqnarray*}
    \delta (a) \int_{N (F)} \Phi_{w_0}^{\tau} (w_0 an) \psi (an a^{- 1})
    \hspace{0.25em} dn & = & \\
    \delta (a) \cdot \delta^{1 / 2} \tau (w_0 aw_0^{- 1}) \int_{N (F)}
    \Phi_{w_0}^{\tau} (w_0 n) \psi (ana^{- 1}) \hspace{0.25em} dn. &  & 
  \end{eqnarray*}
  Since $\delta^{1 / 2} (w_0 aw_0^{- 1}) = \delta^{- 1 / 2} (a)$,
  \[ \delta (a) \cdot \delta^{1 / 2} \tau (w_0 aw_0^{- 1}) = \delta^{1 / 2}
     (a) (\tau w_0) (a) = \delta^{1 / 2} (a) \boldsymbol{z}^{w_0 \lambda} . \]
  By Lemma~\ref{wellknow}, $\Phi_{w_0}^{\tau} (w_0 n) = 1$ if $n \in N
  ( \mathfrak{o})$ and $0$ otherwise. Since $\lambda$ is dominant, $\psi
  (ana^{- 1}) = 1$ when $n \in N ( \mathfrak{o})$, and the statement follows.
\end{proof}

Let $\tau = \tau_{\boldsymbol{z}}$ and define
\begin{equation}
  \label{caltau} C_{\alpha} (\tau) = \frac{1 - q^{- 1} \boldsymbol{z}^{\alpha}}{1
  - \boldsymbol{z}^{\alpha}} .
\end{equation}

\begin{proposition}
  We have
  \begin{equation}
    \label{wwithacomp} \mathcal{W}_{\tau w} \mathcal{A}_w^{\tau} =
    \prod_{\tmscript{\begin{array}{c}
      \alpha \in \Delta^+\\
      w^{- 1} \alpha \in \Delta^-
    \end{array}}} \frac{1 - q^{- 1} \boldsymbol{z}^{- \alpha}}{1 -
    \boldsymbol{z}^{\alpha}} \mathcal{W}_{\tau} .
  \end{equation}
\end{proposition}

\begin{proof}
  This is Casselman and Shalika~{\cite{CasselmanShalika}}, Proposition~4.3.
\end{proof}

In (\ref{demdef}) we defined the Demazure operator $\partial_i =
\partial_{\alpha_i}$ corresponding to a simple reflection $\alpha_i$. We will
also make use of alternative version defined by
\[ \partial_{\alpha_i}' f ( \boldsymbol{z}) = \partial_i' f ( \boldsymbol{z}) =
   \frac{f ( \boldsymbol{z}) - \boldsymbol{z}^{\alpha_i} f (s_i \boldsymbol{z})}{1 -
   \boldsymbol{z}^{\alpha_i}} = \frac{f (s_i \boldsymbol{z}) - \boldsymbol{z}^{- \alpha_i}
   f ( \boldsymbol{z})}{1 - \boldsymbol{z}^{-\alpha_i}} . \]
Again we may define $\partial'_w = \partial_{i_1}' \cdots \partial'_{i_k}$ if
$w = s_{i_1} \cdots s_{i_k}$ is a reduced decomposition. If $w = w_0$ and
$\mu$ is dominant, then $\partial'_{w_0} ( \boldsymbol{z}^{w_0 \mu})$ is the
character of the irreducible representation of highest weight $\mu$.

\section{Whittaker functions}

In this section we will prove that the Iwahori Whittaker functions $W
(\Phi_w)$ are polynomials in $q^{- 1}$ whose constant terms are Demazure
characters.

\begin{proposition}
  \label{casselslick}Let $s = s_{\alpha}$ with $\alpha$ a simple root. Then
  \begin{equation}
    \label{secondid} \mathcal{A}_s \Phi^{\tau s}_w + C_{\alpha} (\tau)
    \Phi^{\tau}_{w_{}} = \left\{ \begin{array}{ll}
      \Phi^{\tau}_w + \Phi^{\tau}_{sw} & \text{if $sw < w$,}\\
      q^{- 1} (\Phi^{\tau}_w + \Phi^{\tau}_{sw}) & \text{if $sw > w$.}
    \end{array} \right.
  \end{equation}
\end{proposition}

\begin{proof}
  The identities
  \[ \mathcal{A}_s \Phi^{\tau}_w = \left\{ \begin{array}{ll}
       (C_{\alpha} (\tau) - q^{- 1}) \Phi_w^{\tau s} + \Phi^{\tau s}_{sw} &
       \text{if $sw < w,$}\\
       (C_{\alpha} (\tau) - 1) \Phi_w^{\tau s} + \Phi^{\tau s}_{sw} & \text{if
       $sw > w,$}
     \end{array} \right. \]
  are Casselman~{\cite{Casselman}}, Theorem 3.4. These imply that
  \begin{equation}
    \label{firstid} \mathcal{A}_s \Phi^{\tau}_w - (C_{\alpha} (\tau) - q^{- 1}
    - 1) \Phi^{\tau s}_{w_{}} = \left\{ \begin{array}{ll}
      \Phi^{\tau s}_w + \Phi^{\tau s}_{sw} & \text{if $sw < w$,}\\
      q^{- 1} (\Phi^{\tau s}_w + \Phi^{\tau s}_{sw}) & \text{if $sw > w$.}
    \end{array} \right.
  \end{equation}
  Replacing $\tau$ by $\tau s$, noting that $C_{\alpha} (\tau s) = C_{-
  \alpha} (\tau)$ and using
\begin{equation*}
  C_{\alpha} (\tau) + C_{- \alpha} (\tau) = 1 + q^{- 1} .
\end{equation*}
we obtain (\ref{secondid}).
\end{proof}

\begin{proposition}
  \label{demazuremi}Let $\alpha$ be a simple root, and $s = s_{\alpha}$ the
  corresponding reflection. Then
  \begin{equation}
    \label{altqmagic} \mathcal{W}_{\tau} ( \mathcal{A}_s \Phi^{\tau s}_w + C_{\alpha}
    (\tau) \Phi^{\tau}_{w_{}}) = (1 - q^{- 1} \boldsymbol{z}^{\alpha})
    \partial_{\alpha}' \mathcal{W}_{\tau} \Phi_w^{\tau} .
  \end{equation}
\end{proposition}

\begin{proof}
  By (\ref{wwithacomp}), we have
  \[ \mathcal{W}_{\tau s} \mathcal{A}_s^{\tau} = \frac{1 - q^{- 1} \boldsymbol{z}^{-
     \alpha}}{1 - \boldsymbol{z}^{\alpha}} \mathcal{W}_{\tau}, \]
  which we rewrite as
  \[ \mathcal{W}_{\tau} \mathcal{A}_s^{\tau s} = \frac{1 - q^{- 1}
     \boldsymbol{z}^{\alpha}}{1 - \boldsymbol{z}^{- \alpha}} \mathcal{W}_{\tau s} . \]
  Now the left-hand side in (\ref{altqmagic}) equals
  \[ \frac{1 - q^{- 1} \boldsymbol{z}^{\alpha}}{1 - \boldsymbol{z}^{- \alpha}} \mathcal{W}_{\tau
     s} \Phi^{\tau s}_w + \frac{1 - q^{- 1} \boldsymbol{z}^{\alpha}}{1 -
     z^{\alpha}} \mathcal{W}_{\tau} \Phi^{\tau}_w = (1 - q^{- 1} \boldsymbol{z}^{\alpha})
     \frac{1}{1 - \boldsymbol{z}^{\alpha}} (\mathcal{W}_{\tau} \Phi^{\tau}_w -
     \boldsymbol{z}^{\alpha} \mathcal{W}_{\tau s} \Phi^{\tau s}) . \]
  By the definition of the Demazure operator we have
  \[ \frac{1}{1 - \boldsymbol{z}^{\alpha}} (\mathcal{W}_{\tau} \Phi^{\tau}_w -
     \boldsymbol{z}^{\alpha} \mathcal{W}_{\tau s} \Phi^{\tau s}) = \partial'_{\alpha}
     \mathcal{W}_{\tau} \Phi^{\tau}_w, \]
  and we are done.
\end{proof}

Define operators on the ring $\mathcal{O}(\hat T)$ of functions on $\hat{T}(\mathbb{C})$
as follows. Let $s_i$ be a simple reflection corresponding to the simple root $\alpha=\alpha_i$.
Define
\begin{equation}
\label{frakdpdef}
\mathfrak{D}'_i=\mathfrak{D}'_\alpha=(1 - q^{- 1} \boldsymbol{z}^{\alpha})\partial_{\alpha}',
\qquad\mathfrak{T}'_i=\mathfrak{T}'_\alpha=\mathfrak{D}'_i-1.
\end{equation}

\begin{proposition}
  \label{earlymagic}
  Suppose that $s = s_{\alpha}$ is a simple reflection and $s w < w$. Then
  \[ \mathcal{W}_{\tau} \Phi_{s w}^{\tau} = \mathfrak{T}'_\alpha\mathcal{W}_{\tau} \Phi^{\tau}_w .\]
\end{proposition}

\medbreak
\begin{proof}
  This follows immediately from Propositions~\ref{demazuremi}
  and~\ref{casselslick}.
\end{proof}

The operator $\mathfrak{T}_i'$ is closely related to the Demazure-Lusztig operator.
We will return to this point in Section~\ref{HeckeAlgebras}. At the moment,
we make use of the Whittaker function to give a short proof that they
satisfy the braid relation. Let $\mathcal{D}$ be the ring of expressions of
the form $\sum_{w \in W} f_w \cdot w$ where $f_w \in \mathcal{O}( \hat{T})$,
and the multiplication is defined by $(f_1 \cdot w_1) (f_2 \cdot w_2) = f_1
{^{w_1} f_2} \cdot w_1 w_2$.  The $\mathfrak{D}_i$ are naturally elements of
this ring. The ring $\mathcal{D}$ acts on $\mathcal{O}( \hat{T})$ in the
obvious way.

\begin{lemma}
  \label{dominantenough}
  Suppose that $D \in \mathcal{D}$ and that $D$ annihilates
  $\boldsymbol{z}^{w_0\lambda}$ for every dominant weight $\lambda$. Then $D = 0$.
\end{lemma}

\begin{proof}
  Define the {\tmem{support}} of $f \in \mathcal{O}( \hat{T})$ to be the
  finite set of weights with nonzero coefficients in $f$. Let $D = \sum_{w \in
  W} f_w \cdot w$. We may choose the dominant weight $\lambda$ so that the
  functions $f_w \boldsymbol{z}^{ww_0 \lambda}$ have disjoint support. Then
  $D\boldsymbol{z}^{w_0\lambda} = 0$ implies that each $f_w = 0$ and so $D = 0$.
\end{proof}

\begin{proposition}\label{whitbr}
Let $s=s_i$ and $s_j$ be simple reflections. Then the operators
$\mathfrak{T}'_i$ and $\mathfrak{T}'_j$ satisfy the same braid relations
as $s_i$ and $s_j$. That is, if $k$ is the order of $s_is_j$ then
\begin{equation}
\label{braidp}
\mathfrak{T}'_i\mathfrak{T}'_j\mathfrak{T}'_i\cdots = 
\mathfrak{T}'_j\mathfrak{T}'_i\mathfrak{T}'_j\cdots,
\end{equation}
where $k$ is the number of factors on both sides of this
equation.
\end{proposition}

\begin{proof}
By Lemma~\ref{dominantenough} it is enough to show that these
both have the same effect on $\boldsymbol{z}^{w_0\lambda}$ where
$\lambda$ is a dominant weight. By Proposition~\ref{wflongev},
$\boldsymbol{z}^{w_0\lambda}=c\mathcal{W}_\tau\Phi_{w_0}(a_\lambda)$
where the constant $c=(\delta^{1/2}\tau)(a_\lambda)$ is independent
of $\boldsymbol{z}$ and hence commutes with operators in $\mathcal{D}$.
Applying either side of (\ref{braidp}) to
$\mathcal{W}_\tau\Phi_{w_0}(a_\lambda)$ gives
$\mathcal{W}_\tau\Phi_{ww_0}(a_\lambda)$ where $w$ is the long
element of the rank two Weyl group $\left<s_i,s_j\right>$.
Indeed, Proposition~\ref{earlymagic} is applicable taking
$(s,w)=(s_i,w_0)$, $(s,w)=(s_j,s_iw_0)$, etc.\ until we
reach $ww_0$ since at each stage $s$ is a descent of $w$.
\end{proof}

\begin{proposition}
\label{quadrel}
The operators satisfy the quadratic relations
\begin{equation}
\label{quadrelation}
(\mathfrak{D}'_i)^2=(1+q^{-1})\mathfrak{D}_i',\qquad
(\mathfrak{T}'_i)^2=(q^{-1}-1)\mathfrak{T}_i'+q^{-1}.
\end{equation}
\end{proposition}

\begin{proof}
  The two relations are equivalent. We prove the first.
  Writing $\alpha = \alpha_i$,
  \[ \mathfrak{D}_i^2 = (1 - q^{-1}\boldsymbol{z}^{\alpha}) \partial'_i (1 -
     q^{-1}\boldsymbol{z}^{\alpha}) \partial'_i = (1 - q^{-1}\boldsymbol{z}^{\alpha})
     (\partial'_i)^2 - (1 - q^{-1}\boldsymbol{z}^{\alpha}) q^{-1} \partial'_i
     \boldsymbol{z}^{\alpha} \partial'_i \]
  and the quadratic relation in the form $\mathfrak{D}_i^2 = (v +
  1)\mathfrak{D}_i$ follows from the properties $(\partial'_i)^2 = \partial_i'$
  and $\partial'_i \boldsymbol{z}^{\alpha} \partial'_i=-\partial'_i$ of the usual Demazure
  operator.
\end{proof}

If $v$ is either an element of $\mathbb{C}$ or a field containing it,
let $\mathcal{H}_v$ be the complex algebra generated by $T_i$ subject
to the quadratic relations $T_i^2=(v-1)T_i+v$ together with the
braid relations. We see from Propositions~\ref{quadrel} and~\ref{whitbr}
that there is a representation of $\mathcal{H}_{q^{-1}}$ on the
Whittaker model $\mathcal{W}_\tau M(\tau)^J$ in which $T_i$ acts
by $\mathfrak{T}_i'$.

Given any $w\in W$ we may construct an element $\mathfrak{T}_w'$ of
$\mathcal{D}$ as follows. Let $w=s_{i_1}\cdots s_{i_k}$
be any reduced decomposition of $w$ into a product of
simple reflections. Then
$\mathfrak{T}_w'=\mathfrak{T}_{i_1}'\cdots\mathfrak{T}_{i_k}'$.
This is well-defined by Proposition~\ref{whitbr}.

Iwahori and Matsumoto~\cite{IwahoriMatsumoto} showed that
the convolution ring of $J$-bi-invariant functions supported in
$G(\mathfrak{o})$ is isomorphic to $\mathcal{H}_q$ 
and that algebra acts by right convolution on $\mathcal{W}M(\tau)^J$. The
rings $\mathcal{H}_q$ and $\mathcal{H}_{q^{-1}}$ are isomorphic,
and one might wonder whether the action of $\mathcal{H}_{q^{-1}}$
we have just defined is the same. It is not, since the isomorphism
class of the convolution action depends on $\tau$, while the isomorphism
class of the representation that we have just defined does not.

\begin{theorem}
\label{midmagic}
Let $w\in W$ and let $\lambda$ be a dominant weight. Then
\begin{equation}
\mathcal{W}_\tau\Phi_{ww_0}^\tau(a_\lambda)=
\delta^{1/2}(a_\lambda)\mathfrak{T}_w'\boldsymbol{z}^{w_0\lambda}.
\end{equation}
\end{theorem}

\begin{proof}
If $w=1$, this follows from Proposition~\ref{wflongev}. The
general case follows by repeated applications of
Proposition~\ref{earlymagic}.
\end{proof}

In passing from the functions $\mathcal{W}_\tau\Phi_w^\tau$ to
$\mathcal{W}_\tau\widetilde\Phi_w^\tau$, the combinatorics of the Bruhat order
begins to play a role.

\begin{proposition}
  \label{deodchar}Let $s$ be a simple reflection and $w_1, w_2 \in W$.
  
  (i) Assume that $sw_1 < w_1$ and $sw_2 < w_2$. Then $w_1 \leqslant w_2$ if
  and only if $sw_1 \leqslant w_2$ if and only if $sw_1 \leqslant sw_2$.
  
  (ii) Assume that $sw_1 > w_1$ and $sw_2 > w_2$. Then $w_1 \geqslant w_2$ if
  and only if $sw_1 \geqslant w_2$ if and only if $sw_1 \geqslant sw_2$.
\end{proposition}

\begin{proof}
  Part (i) is a well-known property of Coxeter groups, called property $Z (s,
  w_1, w_2)$ by Deodhar~{\cite{Deodhar}}. Note that $w \longmapsto ww_0$ is an
  order reversing bijection of $W$. Applying this gives (ii).
\end{proof}

Suppose that $s = s_{\alpha}$ is a descent of $w \in W$: $sw < w$. Then we
will define
\[ H' (w, s) =\{u \in W|u, su \geqslant w\}. \]
\begin{proposition}
  \label{cofinal}The set $H' (w, s)$ is cofinal in $W$ in the sense that if $u
  \in H' (w, s)$ and $t \geqslant u$ then $t \in H' (w, s)$.
\end{proposition}

\begin{proof}
  We have $t \geqslant u$ with both $u, su \geqslant w$. We wish to show that
  $t \in H' (w, s)$. We may assume without loss of
  generality that $su > u$. For if not, then $su < u$ so $t \geqslant su$.
  Thus interchanging $u$ and $su$ if necessary, we may assume that $su > u$.
  Also without loss of generality, $t > st$ since otherwise both $t, st$ are
  $\geqslant u \geqslant w$ as required. Now taking $w_1 = su$ and $w_2 = t$
  in Proposition~\ref{deodchar}~(i), we see that both $t, st \geqslant u$ and
  so $t \in H' (w, s)$.
\end{proof}

Define an integer-valued function $c'_{w, s}$ on $W$ by
\[ c'_{w, s} (u) = \sum_{\tmscript{\begin{array}{c}
     t \in H' (w, s)\\
     t \leqslant u
   \end{array}}} (- 1)^{l (t) - l (u)} . \]
\begin{theorem}
  \label{magictheorem}Let $\alpha$ be a simple root, and let $s = s_{\alpha}$
  denote the corresponding reflection. Assume that $sw < w$. Then
  \begin{equation}
    \label{firstmagic} \mathcal{W}_{\tau} \tilde{\Phi}^{\tau}_{sw} = (1 - q^{- 1}
    \boldsymbol{z}^{\alpha}) \partial'_{\alpha} \mathcal{W}_{\tau} \tilde{\Phi}^{\tau}_w -
    q^{- 1} \sum_{u \in H' (w, s)} \mathcal{W}_{\tau} \Phi^{\tau}_u .
  \end{equation}
  Equivalently,
  \begin{equation}
    \label{secondmagic} \mathcal{W}_{\tau} \tilde{\Phi}^{\tau}_{sw} = (1 - q^{- 1}
    \boldsymbol{z}^{\alpha}) \partial'_{\alpha} \mathcal{W}_{\tau} \tilde{\Phi}^{\tau}_w -
    q^{- 1} \sum_{u \in H' (w, s)} c'_{w, s} (u) \mathcal{W}_{\tau}
    \tilde{\Phi}^{\tau}_u .
  \end{equation}
\end{theorem}

\begin{proof}
  By Proposition~\ref{demazuremi}
  \[ (1 - q^{- 1} \boldsymbol{z}^{\alpha}) \partial'_{\alpha} \mathcal{W}_{\tau}
     \tilde{\Phi}^{\tau}_w = \sum_{x \geqslant w} \mathcal{W}_{\tau} ( \mathcal{A}_s
     \Phi^{\tau s}_x + C_{\alpha} (\tau) \Phi^{\tau}_{x_{}}) . \]
  We split the sum into two parts according as $sx < x$ or $sx > x$ and use
  Proposition~\ref{casselslick}. We have
  \[ \sum_{\tmscript{\begin{array}{c}
       x \geqslant w\\
       sx < x
     \end{array}}} \mathcal{W}_{\tau} ( \mathcal{A}_s \Phi^{\tau s}_x + C_{\alpha}
     (\tau) \Phi^{\tau}_{x_{}}) = \sum_{\tmscript{\begin{array}{c}
       x \geqslant w\\
       sx < x
     \end{array}}} (\mathcal{W}_{\tau} \Phi_x^{\tau} + \mathcal{W}_{\tau} \Phi_{sx}^{\tau}) . \]
  By Proposition~\ref{deodchar}~(i) with $w_1 = w$ and $w_2 = x$, we see that
  \[ \bigcup_{\tmscript{\begin{array}{c}
       x \geqslant w\\
       sx < x
     \end{array}}} \{x, sx\}=\{u \in W|u \geqslant sw\}. \]
  Therefore this contribution equals $\mathcal{W}_{\tau} \tilde{\Phi}^{\tau}_{sw}$. On
  the other hand, let us consider the contributions from $sx > x$. By
  Proposition~\ref{casselslick} these contribute
  \[ q^{- 1} \sum_{\tmscript{\begin{array}{c}
       x \geqslant w\\
       sx > x
     \end{array}}} (\mathcal{W}_{\tau} \Phi_x^{\tau} + \mathcal{W}_{\tau} \Phi_{sx}^{\tau}) = q^{-
     1} \sum_{u \in H' (w, s)} \mathcal{W}_{\tau} \Phi^{\tau}_u . \]
  This proves (\ref{firstmagic}). By M\"obius inversion (Verma~{\cite{Verma}},
  Deodhar~{\cite{Deodhar}}) we may write
  \[ \Phi_u^{\tau} = \sum_{t \leqslant u} (- 1)^{l (t) - l (u)}
     \tilde{\Phi}_t^{\tau}, \]
  and substituting this gives (\ref{secondmagic}).
\end{proof}

The function $c'_{w, s}$ has a tendency to take on only a few nonzero values.
It vanishes off $H' (w, s)$. This sparseness means there are usually only a
few terms on the right-hand side in (\ref{secondmagic}). For example, in
$A_3$, if we consider the pairs $w,s$ where $s$ is a left descent of $w$, we
find sixteen such pairs where $c_{w,s}'$ is always zero. Thus for these pairs
the identity (\ref{secondmagic}) takes the form
\[ \mathcal{W}_{\tau s} ( \tilde{\Phi}^{\tau s}_{sw}) = (1 - q^{- 1} \boldsymbol{z}^{-
   \alpha}) \partial'_{\alpha} \mathcal{W}_{\tau} ( \tilde{\Phi}^{\tau s}_w) . \]
There are seventeen pairs $(w,s)$ such that $c'_{w, s} (u) \neq 0$ for only
one particular $u$.  Then
\[ \mathcal{W}_{\tau s} ( \tilde{\Phi}^{\tau s}_{sw}) = (1 - q^{- 1} \text{}
   \boldsymbol{z}^{- \alpha}) \partial'_{\alpha} \mathcal{W}_{\tau} ( \tilde{\Phi}^{\tau
   s}_w) - q^{- 1} \mathcal{W}_{\tau s} ( \tilde{\Phi}^{\tau s}_u) . \]
Finally, there are three cases where
\[ \mathcal{W}_{\tau s} ( \tilde{\Phi}^{\tau s}_w) = (1 - q^{- 1} \text{} \boldsymbol{z}^{-
   \alpha}) \partial'_{\alpha} \mathcal{W}_{\tau} ( \tilde{\Phi}^{\tau s}_w) - q^{- 1}
   \mathcal{W}_{\tau s} ( \tilde{\Phi}^{\tau s}_u) - q^{- 1} \mathcal{W}_{\tau s} (
   \tilde{\Phi}^{\tau s}_t) + q^{- 1} \mathcal{W}_{\tau s} ( \tilde{\Phi}^{\tau s}_x) .
\]
These are:
\[ \begin{array}{|l|l|l|l|l|}
     \hline
     w & s & u & t & x\\
     \hline
     s_2 & s_2 & s_1 s_2 & s_3 s_2 & s_1 s_3 s_2\\
     \hline
     s_3 s_1 & s_1 & s_2 s_3 s_1 & s_3 s_2 s_1 & s_2 s_3 s_2 s_1\\
     \hline
     s_3 s_1 & s_3 & s_2 s_3 s_1 & s_1 s_2 s_3 & s_2 s_1 s_2 s_3\\
     \hline
   \end{array} \]

The ring $\mathcal{O}( \hat{T})$ of regular functions on $\hat{T}$ is the
complex group algebra over the weight lattice $\Lambda$ of $\hat{T}$. Let $v$
be an indeterminate, and let $\mathcal{R}=\mathbb{C}[v] \otimes \mathcal{O}(
\hat{T})$. If $P \in \mathcal{R}$ we will denote by $P (\boldsymbol{z}, q^{-
1})$ the image of $P$ under the specialization map $\mathcal{R}
\longrightarrow \mathcal{O}( \hat{T})$ that sends $v$ to~$q^{- 1}$.

\begin{theorem}
  \label{wasthmone}Let $\lambda$ be dominant. There exists an element
  $P_{\lambda, w}$ of $\mathcal{R}$ such that
  \[ \mathcal{W}\tilde{\Phi}_w^{\tau} (a_{\lambda}) = \delta^{1 / 2} (a_{\lambda})
     P_{\lambda, w} ( \boldsymbol{z}, q^{- 1}) . \]
  We have
  \[ P_{\lambda, w} ( \boldsymbol{z}, 0) = \partial_w' \boldsymbol{z}^{w_0 \lambda} .
  \]
\end{theorem}

It should be emphasized that $P_{\lambda, w}$ is independent of $q$ and of the
field $F$.

\begin{proof}
  This may be proved by induction on $l (ww_0)$. If $w = w_0$ then by
  Proposition~\ref{wflongev} we have $P_{\lambda, w_0} ( \boldsymbol{z}, q^{- 1})
  = \boldsymbol{z}^{w_0 \lambda}$, so the assertion is true. In general, we may
  apply (\ref{secondmagic}) to deduce both statements for $sw$ when it is
  known for Weyl group elements $\geqslant w$. Indeed, by induction every term
  on the right-hand side is in $\mathcal{R}$, so $P_{\lambda, sw} \in
  \mathcal{R}$, and specializing $q^{- 1} \longrightarrow 0$ produces simply
  $\partial_i' P_{\lambda, w} = \partial'_i \partial_w' \boldsymbol{z}^{w_0
  \lambda} = \partial'_{sw} \boldsymbol{z}^{w_0 \lambda}$.
\end{proof}

\section{\label{bottsam}Fibers of Partial Bott-Samelson Varieties}

Let $G$ be a complex reductive group, and let $B$ be a Borel subgroup. In the
application to Whittaker functions, $G$ will be the group formerly denoted
$\hat{G} ( \mathbb{C})$, but we suppress the hat in this section.

Let $X = G / B$ be the flag variety. If $w$ is an element of the Weyl group
$W$, let $X_w^{\circ}$ be the image of $BwB$ in $X$. The closure
\[ X_w = \bigcup_{u \leqslant w} X_u^{\circ} \]
is the closed Schubert cell.

Choose a reduced decomposition $\mathfrak{w} = (s_{h_1}, \cdots, s_{h_k})$ of
$w = s_{h_1} \cdots s_{h_k}$ into a product of simple reflections. To define
the Bott-Samelson variety $Z_{\mathfrak{w}}$, let $P_i$ be the parabolic
subgroup generated by $B$ and $s_i$. The group $B^k$ acts on $P_{h_1} \times
\cdots \times P_{h_k}$ on the right by
\begin{equation}
  \label{borelaction} (p_1, \cdots, p_k) (b_1, \cdots, b_k) = (p_1 b_1, b_1^{-
  1} p_2 b_2, \cdots, b_{r - 1} p_r b_r) .
\end{equation}
Then $Z_{\mathfrak{w}}$ is the quotient variety. The multiplication map
$P_{h_1} \times \cdots \times P_{h_k} \longrightarrow G$ induces a rational
map $Z_{\mathfrak{w}} \longrightarrow X_w$ that is a birational equivalence.
Although $X_w$ may be singular $Z_{\mathfrak{w}}$ is always smooth, so this
gives a resolution of the singularities of~$X_w$.

If $s$ is an ascent of $w$ with respect to the Bruhat order then we have a
partial Schubert variety $Z_{s, w}$ which is the quotient
$(P \times X_w) / B$ where if $s = s_i$ then $P = P_i$ and $B$ acts on the
right by $(p, x) \cdot b = (pb^{- 1}, bx)$. It is birationally equivalent to
$X_{sw}$.

We are interested in the map $\mu : Z_{s, w} \longrightarrow
X_{sw}$ from the partial Bott-Samelson variety to the Schubert variety. The
fiber over an open Schubert cell $Y_u$ (with $u \leqslant sw$) is either a
single point or a $\mathbb{P}^1$, and we will find a combinatorial criterion
for these cases.

Let $T$ be a maximal torus of $G$ contained in $B$.  Since the fibers of $\mu$
are constant on Schubert cells $Y_t \subset X_{sw}$ with $t \in W$, it
suffices to study the fiber $\mu^{- 1} (y_t)$, where $y_t \in Y_t^T$ is the
unique $T$-fixed point in the Schubert cell $Y_t$. Since the fiber $\mu^{- 1}
(y_t)$ is either a single point or has dimension 1, it is determined by its
Euler characteristic $\chi (\mu^{- 1} (y_t))$, and this is what we will
compute.

\begin{lemma}
  \label{fixedpointseuler}Let $Y$ be a projective complex algebraic variety
  with a $T$ action whose fixed point set $Y^T$ consists of isolated points.
  Then $\chi (Y) =\#Y^T$.
\end{lemma}

\begin{proof}
  Choosing a regular element $\lambda$ of $\tmop{Hom} ( \mathbb{C}^{\times},
  T)$, it follows that $Y$ has a $\mathbb{C}^{\times}$ action with the same
  fixed point set, that is, $Y^{\mathbb{C}^{\times}} = Y^T$. This $\mathbb{C}^{\times}$
  action defines a Bialynicki-Birula cellular decomposition of $Y$, with cells
  $\{U_x \}_{x \in Y^{\mathbb{C}^{\times}}}$ defined by
  \[ U_x =\{z \in Y \mid \lim_{\varepsilon \to 0} \varepsilon \cdot z = x\},
  \]
  one cell for each $x$. See Bialynicki-Birula, Carrell and McGovern
  {\cite{BialynickiBirula}}. Since the cells are all of even real dimension,
  the Euler characteristic of $Y$ is simply the number of cells - that is, the
  number of fixed points.
\end{proof}

\begin{proposition}
  \label{fiberdet}The fiber of $\mu$ over $Y_u$ is $\mathbb{P}^1$ if and only
  if both $u, su \leqslant w$, and is a point otherwise.
\end{proposition}

\begin{proof}
  In view of Lemma~\ref{fixedpointseuler}, in order to compute $\chi (\mu^{-
  1} (y_t))$, we need to compute the number of fixed points in the set $\mu^{-
  1} (y_t)^T$. This is straightforward since the fixed point set $Z_{s, w}^T$
  equals $\{(u, t) \mid u \in \left<s\right>, \; t \leq
  w\}$ and the map $\mu^T : Z_{s, w}^T \longrightarrow X_{sw}^T$ is
  multiplication $(u, t) \mapsto ut$.  Discussions of these facts may be found
  in many places, usually for ``standard'' Bott-Samelson varieties rather than
  these partial ones. See, for example Brion~{\cite{Brion}}.
  
  Now, from these two facts, we compute $(\mu^T)^{- 1} (y_t)$ for $t \leqslant
  w$. In general, we have
  \[ (\mu^T)^{- 1} (y_t) =\{(u, x) \,|\, \text{$u \in  \left<s\right>$, $x
     \leqslant w$ and $ux = t$}\}. \]
  Since $\left<s\right>$ has order two, there are
  at most two points in $(\mu^T)^{- 1} (y_t)$. One of them is the point
  $y_{(1, t)}$, which is the image of the affine point $(1, t)$ in
  $\mathbb{P}^1$. The other possibility is the $y_{(s, s t)}$; but this point
  is only a point of $Z_{s, w}^T$ if $st \leqslant w$. Thus we conclude that
  $(\mu^T)^{- 1} (y_t)$ is in bijection with the elements $z$ such that $t$
  and $st$ are less than or equal to $w$.
  
  The map $\mu$ is an isomorphism over the big cell $Y_{sw}$. Thus it remains
  to study the fibers over the cells $Y_u$ with $u \leqslant w$.
  
  It now suffices to to show that the fiber over $u$ is a $\mathbb{P}^1$ if
  and only if both $u, su \leqslant w$, and is a point otherwise. Thus we must
  show that the Euler characteristic of $\mu^{- 1} (y_u)$ is equal to 2 if and
  only if both $u, su \leqslant w$, and is equal to one otherwise. By
  Lemma~\ref{fixedpointseuler}, we must show that the cardinality of
  $(\mu^T)^{- 1} (y_u)$ is equal to 2 if and only if both $u, su \leqslant w$,
  and is equal to 1 otherwise. But, as explained above, $(\mu^T)^{- 1} (y_u)$
  is the one element set $y_{(1, u)}$ unless both $u, su \leqslant w$, in
  which case $(\mu^T)^{- 1} (y_u)$ is the 2 element set $\{y_{(1, u)}, y_{(s,
  su)} \}$.
\end{proof}

Let us reformulate this proposition using the notation (\ref{suggested}) from
the introduction. Assuming that $sw > w$, define
\[ H (w, s) =\{u \in W|u, su \leqslant w\}. \]
As in Proposition~\ref{cofinal} the set $H (w, s)$ has the property that if $u
\in H (w, s)$ and $t \leqslant u$ then $t \in H (w, s)$. Define
\[ c_{w, s} (u) = \sum_{\tmscript{\begin{array}{c}
     t \in H (w, s)\\
     t \geqslant u
   \end{array}}} (- 1)^{l (t) - l (u)} . \]
\begin{proposition}
  \label{fibersymbolic}Let $s$ be a left ascent of $w$. Then
  \[ \{y \in X_{sw}\,|\, \text{\rm $\mu^{- 1} (y)$ is nontrivial} \}= \sum_{u \in H
     (w, s)} c_{w, s} (u) X_u . \]
\end{proposition}

\begin{proof}
  Let $y \in X$. Let $t \in W$ such that $y \in Y_t$. Then
  \[ \sum_{\tmscript{\begin{array}{c}
       u \in H (w, s)\\
       y \in X_u
     \end{array}}} c_{w, s} (u) = \sum_{\tmscript{\begin{array}{c}
       u \in H (w, s)\\
       t \leqslant u
     \end{array}}} c_{w, s} (u) . \]
  It follows from M\"obius inversion (Verma~{\cite{Verma}} or
  Deodhar~{\cite{Deodhar}}) that given $y \in X$ that this is $1$ if $t \in H
  (w, s)$ and $0$ otherwise. Thus the statement follows from
  Proposition~\ref{fiberdet}.
\end{proof}

\section{Proof of Theorems~\ref{thmain} and~\ref{qone}}

To prove Theorem~\ref{thmain} we may specialize $v = q^{- 1}$ with $q$ the
cardinality of the residue field. Let $\theta : \hat{T} ( \mathbb{C})
\longrightarrow \hat{T} ( \mathbb{C})$ be the map that sends $\boldsymbol{z} \longmapsto
\boldsymbol{z}^{- 1}$. Then
\begin{equation}
  \label{delconj} \partial_i = \theta  \partial_i' \theta 
\end{equation}
Let $\lambda$ be a dominant weight of $\hat{G} ( \mathbb{C})$. Then $- w_0
\lambda$ is also a dominant weight, and (\ref{bigxdef}) may be written
\begin{equation}
  \label{xdefined} \boldsymbol{X}_w (\lambda) = \delta (a_{- w_0 \lambda})^{- 1
  / 2} \theta \mathcal{W}_\tau\tilde{\Phi}^\tau_{ww_0} (a_{- w_0 \lambda}) ,
\end{equation}
where $\tau=\tau_{\boldsymbol{z}}$. Similarly
\begin{equation}
  \label{ydefined} \boldsymbol{Y}_w (\lambda) = \delta (a_{- w_0 \lambda})^{- 1
  / 2} \theta \mathcal{W}_\tau \Phi_{ww_0}^\tau (a_{- w_0 \lambda}) .
\end{equation}
We have $u \in H (s, w)$ if and only if $uw_0 \in H' (s, ww_0)$. Note that
$sww_0 < ww_0$ so that $H' (s, ww_0)$ is defined, and it is also easy to see
that $c_{w, s} (u) = c_{ww_0, s}' (uw_0)$.

\begin{proposition}
  \label{correctionterms}Let $\lambda$ be a dominant weight. Then
  \begin{equation}
  \label{wehave}
  \boldsymbol{X}_1 (\lambda) =\boldsymbol{Y}_1 (\lambda)
     =\boldsymbol{z}^{\lambda} .
  \end{equation}
  Assume that $s = s_{\alpha}$ is a simple reflection and that $sw > w$. Then
  \begin{equation}
    \label{schubid} \boldsymbol{X}_{sw} (\lambda) = (1 - q^{- 1} \boldsymbol{z}^{-
    \alpha}) \partial_{\alpha} \boldsymbol{X}_w (\lambda) - q^{- 1}
    \sum_{\tmscript{\begin{array}{c}
      u \in H (w, s)
    \end{array}}} c_{w, s} (u)\boldsymbol{X}_u (\lambda) .
  \end{equation}
\end{proposition}

\begin{proof}
  Equation~(\ref{wehave}) follows from Proposition~\ref{wflongev}.
  To prove (\ref{schubid})
  we begin with (\ref{secondmagic}), with $w$ replaced by $ww_0$. Applying $\theta $
  and using (\ref{delconj}) we obtain
  \[ \theta \mathcal{W}_{\tau} \tilde{\Phi}^{\tau}_{sww_0} = (1 - q^{- 1} \boldsymbol{z}^{-
     \alpha}) \partial_{\alpha} \theta \mathcal{W}_{\tau} \tilde{\Phi}^{\tau}_{ww_0} - q^{- 1}
     \sum_{uw_0 \in H' (ww_0, s)} c'_{ww_0, s} (uw_0) \theta \mathcal{W}_{\tau}
     \tilde{\Phi}^{\tau}_u . \]
  Now evaluating this at $a_{- w_0 \lambda}$ and multiplying by the constant
  $\delta (a_{- w_0 \lambda})^{- 1 / 2}$ (which depends on $\lambda$ but not
  $\boldsymbol{z}$) we obtain (\ref{schubid}).
\end{proof}

Combining Propositions~\ref{fibersymbolic} and~\ref{correctionterms} gives
Theorem~\ref{thmain}.

To prove Theorem~\ref{qone}, we will take $q = 1$. Then
the operators $\mathfrak{D}_i$ simplify to
\[ \mathfrak{D}_i f ( \boldsymbol{z}) = f (z) - \boldsymbol{z}^{- \alpha_i} f (s_i
   \boldsymbol{z}) . \]
These operators do not satisfy the braid relation. We show that if we define
$\widehat{\boldsymbol{X}}_w ( \boldsymbol{z})$ to be the right-hand side of
(\ref{permutahedron}), and if $s = s_i$ with $l (sw) > l (w)$, then
\begin{equation}
  \label{xprimerecursion} \widehat{\boldsymbol{X}}_{sw} = \mathfrak{D}_i
  \widehat{\boldsymbol{X}}_w - \sum_{u \in H (w, s)} c (u) \widehat{\boldsymbol{X}}_u,
\end{equation}
where $c (u)$ is as in Theorem~\ref{thmain}. The dominant weight $\lambda$
will be fixed and we suppress it from the notation. By comparison with 
Theorem~\ref{thmain}, $\boldsymbol{X}_w$ (with $q = 1$) and
$\widehat{\boldsymbol{X}}_w$ both satisfy the same recursion, and agree when $w =
1$; therefore by induction on the Bruhat order, they are equal.

Let $\widehat{\boldsymbol{Y}}_w = \boldsymbol{z}^{w (\rho + \lambda) -
\rho}$. The identity (\ref{xprimerecursion}) is equivalent by M\"obius
inversion to
\begin{equation}
  \label{equivdc} \mathfrak{D}_i \widehat{\boldsymbol{X}}_w - \sum_{u \in H (w, s)}
  \widehat{\boldsymbol{Y}}_u,
\end{equation}
Since $\boldsymbol{z}^{- \alpha} \boldsymbol{z}^{- s \rho} = \boldsymbol{z}^{- s
\rho}$, we have
\[ \mathfrak{D}_i \widehat{\boldsymbol{Y}}_u =
   \widehat{\boldsymbol{Y}}_u - \widehat{\boldsymbol{Y}}_{su} .
\]
Therefore
\begin{equation}
  \label{possdc} \mathfrak{D}_i \widehat{\boldsymbol{X}}_w = \sum_{u \leqslant w} (-
  1)^{l (u)} \widehat{\boldsymbol{Y}}_u + \sum_{su \leqslant w} (-
  1)^{l (u)} \widehat{\boldsymbol{Y}}_u,
\end{equation}
where in the second term we have replaced $u$ by $su$.

By Proposition~\ref{deodchar} (ii), we have $u \leqslant sw$ if and only $\min
(u, su) \leqslant w$ where the notation $\min (u, su)$ makes sense since $u$
and $su$ are always comparable in the Bruhat order. So
\[ \{u|u \leqslant sw\}=\{u|u \leqslant w\} \cup \{u|su \leqslant w\} \]
and we may write
\[ \widehat{\boldsymbol{X}}_{sw} = \sum_{u \leqslant w} (- 1)^{l (u)}
   \widehat{\boldsymbol{Y}}_u + \sum_{su \leqslant w} (- 1)^{l (u)}
   \widehat{\boldsymbol{Y}}_u - \sum_{u, su \leqslant w} (- 1)^{l (u)}
   \widehat{\boldsymbol{\boldsymbol{Y}}}_u \]
where we have subtracted the terms that are double counted in the first two
sums. Now using (\ref{possdc}) we obtain (\ref{equivdc}).

\section{\label{HeckeAlgebras}Hecke algebras}

\,The space of Iwahori fixed vectors $M (\tau)^J$ for a fixed unramified
character $\tau$ is isomorphic to the finite-dimensional Hecke algebra
$\mathcal{H}$ of compactly supported $J$-bi-invariant functions on $G (F)$
which have support inside of $G (\mathfrak{o})$. We recall the isomorphism,
following Bump and Nakasuji~{\cite{BumpNakasuji}}, who in turn
follow Rogawski~{\cite{Rogawski}}.

In this isomorphism, $f \in M (\tau)^J$ corresponds to the element
$\varrho_{\tau} (f) \in \mathcal{H}$ where
\[ \varrho_{\tau} (f) (g) = \left\{ \begin{array}{ll}
     f (g^{- 1}) & \text{if $g \in G (\mathfrak{o})$,}\\
     0 & \text{otherwise.}
   \end{array} \right. \]
Both $M (\tau)^J$ and $\mathcal{H}$ are left $\mathcal{H}$-modules, where the
action of $\mathcal{H}$ on $M (\tau)$ is
\[ \phi \cdot f (g) = \int_{G (\mathfrak{o})} \phi (x) f (g x) \, d x. \]
Then $\varrho_{\tau}$ is a homomorphism of left $\mathcal{H}$-modules.

The Hecke algebra $\mathcal{H}$ has the following description by generators
and relations, due to Iwahori and Matsumoto {\cite{IwahoriMatsumoto}}. If
$t_w$ is the characteristic function of the double coset $J w J$ then
$\mathcal{H}$ is generated by the $t_i = t_{s_i}$ where $s_i$ is a simple
reflection. The generators satisfy the quadratic relation
\[ t_i^2 = (q - 1) t_i + q, \]
together with the braid relations. Thus $t_i=t_{s_i}$.
The braid relations and the quadratic relations give a presentation of
$\mathcal{H}$ which is valid even if $q$ is not a prime.

Let $v$ be an indeterminate. Let $\mathcal{H}_v$ be the abstract algebra over
$\mathbb{C}[v,v^{-1}]$ generated by $T_1, \cdots, T_r$ subject to the braid
relations together with the condition that
\[ T_i^2 = (v - 1) T_i + v. \]
There is an antihomomorphism $\sigma:\mathcal{H}_v\to\mathcal{H}$ 
defined by $\sigma(v)=q^{-1}$ and $\sigma(T_w)=t_w^{-1}$. In particular
\begin{equation}
\label{sigmat}
\sigma(T_i+1)=q^{-1}(t_i+1).
\end{equation}
The operators $\mathfrak{T}_i=\mathfrak{D}_i - 1$ are closely related to
the \textit{Demazure-Lusztig operators} that were introduced in
Lusztig~\cite{LusztigK1} equation (8.1). These are the operators defined by
\[ \mathcal{L}_i (\boldsymbol{z}^{\lambda}) =\mathcal{L}_{i, v}
   (\boldsymbol{z}^{\lambda}) = \frac{\boldsymbol{z}^{\lambda} -\boldsymbol{z}^{s_i
   \lambda}}{\boldsymbol{z}^{\alpha_i} - 1} - v \frac{\boldsymbol{z}^{\lambda}
   -\boldsymbol{z}^{\alpha + s_i \lambda}}{\boldsymbol{z}^{\alpha_i} - 1} . \]
Lusztig shows that these satisfy the same relations as the $T_i$ in the
(finite) Hecke algebra, and we will also prove this in the equivalent
form of Proposition~\ref{braidrel} below. The precise relationship
between our operators and Lusztig's is as follows. Lusztig's operators
$\mathcal{L}_i$ satisfy the quadratic relation
\begin{equation}
\label{dlrel}
\mathcal{L}_i^2 = (v - 1)\mathcal{L}_i + v.
\end{equation}
Replacing $v$ by $v^{- 1}$ and multiplying by $- v$ let $\mathcal{L}_i' = -
v\mathcal{L}_{i, v^{- 1}}$. We have
\begin{equation}
\label{dlprel}
\mathcal{L}_i' \boldsymbol{z}^{\lambda} = \frac{- v\boldsymbol{z}^{\lambda}
   -\boldsymbol{z}^{\alpha_i + s_i \lambda} + v\boldsymbol{z}^{s_i \lambda}
   +\boldsymbol{z}^{\lambda}}{\boldsymbol{z}^{\alpha_i} - 1} .
\end{equation}
It follows from (\ref{dlrel}) that the $\mathcal{L}_i'$ satisfy the same braid
and quadratic relations as the $\mathcal{L}_i$ and hence there is a
representation of $\mathcal{H}_v$ in which $T_i \longrightarrow \mathcal{L}_i'$.
We recall that if $\lambda$ is a weight, then $\zeta^\lambda$ is the
operation on $\mathcal{O}(\hat{T})$ defined by multiplication by $\boldsymbol{z}^{-\lambda}$.
It follows easily from (\ref{dlprel}) that 
\begin{equation}
\label{usvsthem}
\zeta^{- \rho} (\mathfrak{D}_i - 1) \zeta^{\rho} =\mathcal{L}'_i .
\end{equation}
Thus the following proposition is equivalent to the result of~\cite{LusztigK1}.
Our proof is different, since we use Whittaker functions to prove
the braid relations.

\begin{proposition}
\dueto{Lusztig~\cite{LusztigK1}}
\label{braidrel}
  The operators $\mathfrak{T}_i =\mathfrak{D}_i - 1$ satisfy the braid
  relations and the quadratic relations of $\mathcal{H}_v$. Thus $T_i
  \longmapsto \mathfrak{T}_i$ is a representation of~$\mathcal{H}_v$.
\end{proposition}

\begin{proof}
The corresponding facts for $\mathfrak{T}'_i$, which is 
defined by (\ref{frakdpdef}), are Proposition~\ref{quadrel}
for the quadratic relation and Proposition~\ref{whitbr}
for the braid relation. Since
$\mathfrak{T}_i=\theta\mathfrak{T}_i'\theta$, the result
follows.
\end{proof}

We have now defined an action of $\mathcal{H}_v$ on
$\mathbb{C}[v,v^{-1}]\otimes \mathcal{O}(\hat{T})$, so Theorem~\ref{amazing}
now has meaning. In order to prove it we will have define certain
maps $\xi_\lambda:\mathcal{H}_v\to \mathbb{C}[v,v^{-1}]\otimes \mathcal{O}(\hat{T})$
and that is our next goal. These maps $\xi_\lambda$ also turn out to be the
key to extend the action to the affine Hecke algebra
$\widetilde{\mathcal{H}}_v$.

\begin{lemma}
  \label{bylemma}
  Let $s = s_{\alpha}$ be a simple reflection, and let $A^{\tau s}_s :
  \mathcal{H} \longrightarrow \mathcal{H}$ be the homomorphism of left
  $\mathcal{H}$-modules that makes the following diagram commutative:
\[\begin{CD}
M(\tau s)@>\mathcal{A}_s^{\tau s}>>M(\tau)\\
@VV\varrho_{\tau s}V @VV\varrho_{\tau}V\\
\mathcal{H}@>A_s^{\tau s}>> \mathcal{H}
\end{CD}\]
  Then $A_s^{\tau s}$ is right multiplication by
  \begin{equation}
    \label{righttrans} \frac{1}{q} t_i + \left( 1 - \frac{1}{q} \right)
    \frac{\boldsymbol{z}^{- \alpha}}{1 -\boldsymbol{z}^{- \alpha}},
  \end{equation}
  where $\tau = \tau_{\boldsymbol{z}}$.
\end{lemma}

\begin{proof}
  We recall that $\varrho_\tau$ is a homomorphism of left $\mathcal{H}$-modules.
  The composition $\varrho_{\tau} \circ \mathcal{A}_s^{\tau s} \circ
  \varrho_{\tau s}^{- 1}$ is thus a homomorphism of left $\mathcal{H}$-modules
  $\mathcal{H} \longrightarrow \mathcal{H}$ and since $\mathcal{H}$ is a ring
  it must consist of right multiplication by some element. The scalar may be
  evaluated by applying $\varrho_{\tau} \circ \mathcal{A}_s^{\tau s} \circ
  \varrho_{\tau s}$ to the unit element of $\mathcal{H}$, which corresponds
  under $\varrho^{- 1}$ to $\Phi_1^{\tau s}$. By Casselman~{\cite{Casselman}},
  Theorem 3.4 we have
  \[ \mathcal{A}_s^{\tau} \Phi^{\tau}_1 = \frac{1}{q} \Phi^{\tau s}_s + \left(
     1 - \frac{1}{q} \right) \frac{\boldsymbol{z}^{\alpha}}{1
     -\boldsymbol{z}^{\alpha}} \Phi^{\tau s}_1 \]
  so
  \[ \mathcal{A}_s^{\tau s} \varrho_{\tau s}^{- 1} (1) =\mathcal{A}_s^{\tau}
     \Phi^{\tau s}_1 = \frac{1}{q} \Phi^{\tau}_s + \left( 1 - \frac{1}{q}
     \right) \frac{\boldsymbol{z}^{- \alpha}}{1 -\boldsymbol{z}^{- \alpha}}
     \Phi^{\tau}_1.\]
  Applying $\rho_\tau$ gives
  \[\frac{1}{q} t_i + \left( 1 -
     \frac{1}{q} \right) \frac{\boldsymbol{z}^{- \alpha}}{1 -\boldsymbol{z}^{-
     \alpha}} . \]
  
\end{proof}

Let us fix a dominant weight $\lambda$ for $\hat{T}$. Then we will define a
map $\xi'_{\lambda} : \mathcal{H} \longrightarrow
\mathcal{O}( \hat{T})$, where $\mathcal{O}( \hat{T})$ is the ring of rational
functions on $\hat{T}$, which is the group algebra of the weight lattice
$\Lambda$ of $\hat{T}$. We
regard a rational function on $\hat{T}$ as a function on the complex points,
and to specify $\xi_{\lambda} (\phi)$ for $\phi \in \mathcal{H}$, it is enough
to specify its value $\xi_{\lambda} (\phi) (\boldsymbol{z})$ for $\boldsymbol{z}
\in \hat{T} (\mathbb{C})$. Let $\tau = \tau_{\boldsymbol{z}}$ and $\Phi =
\varrho_{\tau}^{- 1} (\phi) \in M (\tau)$. 
Define $\xi_{\lambda}' (\phi)$ by
\[ \xi'_{\lambda} (\phi) (\boldsymbol{z}) = \delta^{- 1 / 2} (a_{\lambda}) \,
   \mathcal{W}_{\tau} \Phi (a_{\lambda}) . \]

\begin{proposition}
  The following diagram is commutative for any dominant weight $\lambda$ of
  $\hat{T}$. Let $s = s_i$ be a simple reflection, and let $\alpha = \alpha_i$
  be the corresponding simple root.
  \begin{equation}
  \label{primehocd}
   \begin{CD}
   \mathcal{H} @>\xi'_{\lambda}>> \mathcal{O}( \hat{T})\\
      @VV\times q^{-1}(1+t_i)V @VV(1-q^{-1}\boldsymbol{z}^\alpha)\partial_\alpha' V\\
  \mathcal{H} @>\xi'_{\lambda}>> \mathcal{O}( \hat{T})
  \end{CD}
  \end{equation}
  where the left vertical arrow is right multiplication by $q^{-1} (1 + t_i)$. Moreover
  \begin{equation}
    \label{xprimeev} \xi'_{\lambda} (t_{w^{- 1}}) = \delta^{-1/2}(a_\lambda)\,\mathcal{W}_{\tau} \Phi^{\tau}_w
    (a_{\lambda}) .
  \end{equation}
\end{proposition}

\medbreak
\begin{proof}
  Let us consider both ways of traversing the commutative diagram
  (\ref{primehocd}) as applied
  to $t_{w^{- 1}}\in\mathcal{H}$. Then with
  $\tau = \tau_{\boldsymbol{z}}$ we have $\varrho_{\tau}^{- 1} t_{w^{- 1}} =
  \Phi^{\tau}_w$ and so we obtain (\ref{xprimeev}). Furthermore by
  Proposition~\ref{demazuremi}
  \[ \delta^{1 / 2} (a_{\lambda}) (1 - q^{- 1} \boldsymbol{z}^{\alpha})
     \partial'_{\alpha} \xi_{\lambda}' (t_{w^{- 1}}) = \mathcal{W}_{\tau}
     (\mathcal{A}_s^{\tau s} \Phi^{\tau s}_w + C_{\alpha} (\tau)
     \Phi^{\tau}_w) (a_{\lambda}) . \]
  By Lemma~\ref{bylemma},
  \[ \mathcal{A}_s^{\tau s} \Phi^{\tau s}_w =\mathcal{A}_s^{\tau s}
     \varrho_{\tau s}^{- 1} t_{w^{- 1}} = \varrho_{\tau}^{- 1} \left( t_{w^{-
     1}} \left( \frac{1}{q} t_i + (1 - q^{- 1}) \frac{\boldsymbol{z}^{-
     \alpha}}{1 -\boldsymbol{z}^{- \alpha}} \right) \right) . \]
  We have
  \[ \left( 1 - \frac{1}{q} \right) \frac{\boldsymbol{z}^{- \alpha}}{1
     -\boldsymbol{z}^{- \alpha}} + C_{\alpha} (\tau) = \frac{1}{q} . \]
  Therefore
  \[ \mathcal{A}_s^{\tau s} \Phi^{\tau s}_w + C_{\alpha} (\tau) \Phi^{\tau}_w
     = \varrho_{\tau}^{- 1} \left( t_{w^{- 1}} q^{- 1} (t_i + 1) \right) . \]
  Applying $\mathcal{W}_{\tau}$ and evaluating at $a_{\lambda}$, we see that
  \[ (1 - q^{- 1} \boldsymbol{z}^{\alpha}) \partial'_{\alpha} \xi_{\lambda}'
     (t_{w^{- 1}}) = \xi_{\lambda}' (t_{w^{- 1}} q^{- 1} (t_i + 1)), \]
  which is the required commutativity.
\end{proof}

If $\lambda$ is dominant, we will define $\xi_\lambda$ as
follows. By Theorem~\ref{qone}, if $\phi\in\mathcal{H}_v$ then
$\theta\xi'_{-w_0\lambda}(\sigma\phi)$ is a polynomial in $q^{-1}$, so there exists
an element $\xi_\lambda(\phi)$ of $\mathbb{C}[v,v^{-1}] \otimes
\mathcal{O}(\hat{T})$ that corresponds to it. Thus $\xi_{\lambda} :
\mathcal{H}_v \longrightarrow \mathbb{C}[v,v^{-1}] \otimes \mathcal{O}(
\hat{T})$ is the map such that
\[ \hbox{ev}\circ\xi_\lambda = \theta\circ\xi'_{-w_0\lambda}\circ\sigma\;.\]
where $\hbox{ev}:\mathbb{C}[v,v^{-1}]\otimes\mathcal{O}(\hat{T})\to \mathcal{O}(\hat{T})$
is the evaluation map that sends $v$ to~$q^{-1}$.

\begin{proposition}
  \label{commdomwt}Let $\lambda$ be a weight, and let $\alpha=\alpha_i$ be a simple
  root The following diagram is commutative:
  \begin{equation}
  \label{xicomdiag}
  \begin{CD}
     \mathcal{H}_v @>\xi_{\lambda}>> \mathbb{C}[v,v^{-1}]\otimes \mathcal{O}(\hat{T})\\
     @VV (1+T_i)\times\cdot V @VV\mathfrak{D}_i V\\
   \mathcal{H}_v @>\xi_{\lambda}>> \mathbb{C}[v,v^{-1}]\otimes \mathcal{O}(\hat{T})
  \end{CD}
  \end{equation}
  where the left vertical arrow is left multiplication by $1 + T_i$. 
  Moreover
  \begin{equation}
    \xi_{\lambda} (T_{(w w_0)^{- 1}}^{- 1}) =\boldsymbol{Y}_w (\lambda) .
    \label{xilambdat}
  \end{equation}
\end{proposition}

\begin{proof}
  We observe that (\ref{xilambdat}) follows from the definitions of
  $\xi_{\lambda}$ and $\boldsymbol{Y}_w (\lambda)$ and (\ref{xprimeev}). As for
  (\ref{xicomdiag}), consider the diagram
\begin{equation*}
\begin{CD}
\mathcal{H}_v@>{\qquad\sigma\qquad}>>\mathcal{H}@>\xi'_{-w_0\lambda}>>\mathcal{O}(\hat T)@>{\qquad\theta\qquad}>>\mathcal{O}(\hat T)\\    
@V(1+T_i)\times VV @V q^{-1}(1+t_i)VV @VV(1-q^{-1}\boldsymbol{z}^\alpha)\partial'_\alpha V @VV(1-q^{-1}\boldsymbol{z}^{-\alpha})\partial_\alpha V\\
\mathcal{H}_v@>{\qquad\sigma\qquad}>>\mathcal{H}@>\xi'_{-w_0\lambda}>>\mathcal{O}(\hat T)@>{\qquad\theta\qquad}>>\mathcal{O}(\hat T)\\
\end{CD}
\end{equation*}
The commutativity of the left square follows from (\ref{sigmat}).
The commutativity of the middle square is (\ref{primehocd}).
The commutativity of the right square is the fact that $\theta \partial_\alpha \theta = \partial_\alpha'$.
The composition on the top is $\hbox{ev}\circ\xi_\lambda$ and the statement follows.
\end{proof}

The (extended) affine algebra $\tilde{\mathcal{H}}_v$ is generated by
$\mathcal{H}_v$ and another commutative subalgebra $\zeta^{\Lambda}$
isomorphic to the weight lattice $\Lambda$. If $\lambda \in \Lambda$ let
$\zeta^{\lambda}$ be the corresponding element of $\zeta^{\Lambda}$. To
complete the presentation of $\tilde{\mathcal{H}}$ we have the relation
\begin{equation}
  \label{bernstein} T_i \zeta^{\lambda} - \zeta^{s_i \lambda} T_i =
  \zeta^{\lambda} T_i - T_i \zeta^{s_i \lambda} = \left( \frac{v - 1}{1 -
  \zeta^{- \alpha_i}} \right) (\zeta^{\lambda} - \zeta^{s_i \lambda}) .
\end{equation}
\begin{lemma}
  Let $\lambda$ run through the dominant weights, and for each $\lambda$ let
  $w$ run through a set of coset representatives for $W / W_{\lambda}$ where
  $W_{\lambda}$ is the stabilizer of $\lambda$. Then $\boldsymbol{Y}_w
  (\lambda)$ runs through a basis of the $\mathbb{C}(v)$-vector space
  $\mathbb{C}(v) \otimes \mathcal{O}( \hat{T})$. Moreover
  \begin{equation}
    \label{xidsdecomp} \mathbb{C}[v,v^{-1}] \otimes \mathcal{O}( \hat{T}) =
    \bigoplus_{\text{$\lambda$ dominant}} \xi_{\lambda} (\mathcal{H}_v) .
  \end{equation}
\end{lemma}

\begin{proof}
  Every weight may be
  written as $w \lambda$ with $\lambda$ dominant. Here $\lambda$ is
  uniquely determined, and if there is more than one choice for $w$
  we take the one of smallest length. We make a partial
  order on the weights by $w_1 \lambda_1 \preccurlyeq w_2 \lambda_2$ with
  either $\lambda_1 < \lambda_2$ in the usual partial order on weights, or
  $\lambda_1 = \lambda_2$ and $w_1 \leqslant w_2$ in the Bruhat order. Then
  $\boldsymbol{Y}_w (\lambda) =\boldsymbol{z}^{w \lambda}$ plus terms that are
  lower in the partial order, so these are a basis. The direct sum
  decomposition follows from~(\ref{xilambdat}).
\end{proof}

\begin{theorem}
  \label{cherednik} There is a right action of
  $\tilde{\mathcal{H}}_v$ on $\mathbb{C}[v,v^{-1}] \otimes \mathcal{O}( \hat{T})$
  such that $\xi_{\lambda}$ is a homomorphism of right
  $\mathcal{H}_v$-modules. In this action
  \begin{equation}
    \label{cheredaction} \zeta^{\lambda} \cdot \boldsymbol{z}^{\mu}
    =\boldsymbol{z}^{\mu - \lambda}, \hspace{2em} (1 + T_i) \cdot
    \boldsymbol{z}^{\mu} =\mathfrak{D}_i (\boldsymbol{z}^{\mu}) = \left( \frac{1 -
    v\boldsymbol{z}^{- \alpha_i}}{1 -\boldsymbol{z}^{- \alpha_i}} \right)
    (\boldsymbol{z}^{\mu} -\boldsymbol{z}^{s_i \mu - \alpha_i}) .
  \end{equation}
\end{theorem}

This action of $\tilde{\mathcal{H}}_v$ appeared previously in
Lusztig~\cite{LusztigK1}. The interpretation here in terms of Whittaker
functions appears to be new.

\medbreak
\begin{proof}
  We make use of (\ref{xidsdecomp}). Each summand may be given a right
  $\mathcal{H}_v$-module structure such that $\xi_{\lambda}$ is a right
  $\mathcal{H}_v$-module homomorphism. Thus $\mathbb{C}[v,v^{-1}] \otimes
  \mathcal{O}( \hat{T})$ becomes an $\mathcal{H}_v$-module. By the
  commutativity of (\ref{xicomdiag})
  \[ (1 + T_i) \cdot \xi_{\lambda} (\phi) = \xi_{\lambda} ((1 + T_i) \phi)
     =\mathfrak{D}_i \xi_{\lambda} (\phi), \]
  which gives the second identity in (\ref{cheredaction}). As for the first,
  we make this a definition and then must check compatibility with the
  relation (\ref{bernstein}) in the equivalent form
  \begin{equation}
    \label{equivbernstein} (T_i + 1) \zeta^{\lambda} - \zeta^{s_i \lambda}
    (T_i + 1) = \left( \frac{v - \zeta^{- \alpha_i}}{1 - \zeta^{- \alpha_i}}
    \right) (\zeta^{\lambda} - \zeta^{s_i \lambda}) .
  \end{equation}
  We will write $s = s_i$ and $\alpha = \alpha_i$. Applying the left-hand side
  of (\ref{equivbernstein}) to $\boldsymbol{z}^{\mu}$ gives
  \[ \left( \frac{1 - v\boldsymbol{z}^{- \alpha}}{1 -\boldsymbol{z}^{- \alpha}}
     \right) (\boldsymbol{z}^{\mu - \lambda} -\boldsymbol{z}^{- \alpha + s \mu - s
     \lambda}) - \left( \frac{1 - v\boldsymbol{z}^{- \alpha}}{1 -\boldsymbol{z}^{-
     \alpha}} \right) (\boldsymbol{z}^{\mu - s \lambda} -\boldsymbol{z}^{- \alpha
     + s \mu - s \lambda}) \]
  This equals
  \[ \left( \frac{1 - v\boldsymbol{z}^{- \alpha}}{1 -\boldsymbol{z}^{- \alpha}}
     \right) (\boldsymbol{z}^{\mu - \lambda} -\boldsymbol{z}^{\mu - s \lambda})
     =\boldsymbol{z}^{\mu} \left( \frac{v -\boldsymbol{z}^{\alpha}}{1
     -\boldsymbol{z}^{\alpha}} \right) (\boldsymbol{z}^{- \lambda}
     -\boldsymbol{z}^{- s \lambda}) . \]
  which equals the right-hand side of (\ref{equivbernstein}) using the action
  $\zeta^{\lambda} \cdot \boldsymbol{z}^{\mu} =\boldsymbol{z}^{\mu - \lambda}$, as
  desired.
\end{proof}

The maps $\xi_\lambda$ have two applications in this paper. The first is
that they allow one to extend the action of $\mathcal{H}_v$ to
all of $\widetilde{\mathcal{H}}_v$ as in the previous theorem. The second
application is Theorem~\ref{amazing}.

\begin{proof}[Proof of Theorem~\ref{amazing}]
  Since $l (w_0) = l (w) + l (w_0 w^{- 1})$ we have $T_{(w w_0)^{- 1}} T_w =
  T_{w_0^{- 1}}$ or $T_{(w w_0)^{- 1}}^{- 1} = T_w T_{w_0^{- 1}}^{- 1} .$
  Using (\ref{xilambdat}) and the fact that $\xi_{\lambda}$ is
  $\mathcal{H}_v$-equivariant,
  \[ \boldsymbol{Y}_w (\lambda) = \xi_{\lambda} (T_{(w w_0)^{- 1}}^{- 1}) = T_w
     \cdot \xi_{\lambda} (T_{w_0^{- 1}}^{- 1}) . \]
  But using Proposition~\ref{correctionterms} and (\ref{xilambdat}) again,
  $\xi_{\lambda} (T_{w_0}^{- 1}) =\boldsymbol{Y}_1 (\lambda)
  =\boldsymbol{z}^{\lambda}$.
\end{proof}

\section{Complements}

It is easy to see that $\partial_i \boldsymbol{z}^{- \rho} = 0$ for each
Demazure operator $\partial_i$. As a consequence, $\mathfrak{D}_i
\boldsymbol{z}^{- \rho} = 0$ and so in the action of Theorem~\ref{cherednik} we
have
\[ T_i \cdot \boldsymbol{z}^{- \rho} = -\boldsymbol{z}^{- \rho} . \]
Hence this vector spans a one-dimensional $\mathcal{H}_v$-invariant subspace
affording the sign representation of $\mathcal{H}_v$.

\begin{theorem}
  The $\tilde{\mathcal{H}}_v$-module $\mathbb{C}[v,v^{-1}] \otimes \mathcal{O}(
  \hat{T})$ is isomorphic to the representation of $\tilde{\mathcal{H}}_v$
  induced from the sign representation of $\mathcal{H}_v$.
\end{theorem}

\begin{proof}
  The module $\mathbb{C}[v,v^{-1}] \otimes \mathcal{O}( \hat{T})$ is generated by
  this vector $\boldsymbol{z}^{- \rho}$ since each coset of $\mathcal{H}_v$ has
  as a coset representation $\zeta^{\lambda}$ for a unique weight $\lambda$,
  which maps $\boldsymbol{z}^{- \rho}$ to a basis vector $\boldsymbol{z}^{-
  \lambda - \rho}$. The statement is thus clear.
\end{proof}

Finally, let us define certain elements $\boldsymbol{C}_w (\lambda)$ that
resemble the $\boldsymbol{X}_w (\lambda)$ but may in some sense be more natural.
We begin by recalling the Kazhdan-Lusztig basis $C_w'$ for the (finite) Hecke
algebra $\mathcal{H}_v$. It is uniquely characterized by the following
properties. First, $\overline{C_w}' = C'_w$, where $x \longmapsto \bar{x}$ is the
Kazhdan-Lusztig involution that sends $q \rightarrow q^{- 1}$ and $T_w
\rightarrow T_{w^{- 1}}^{- 1}$; and second, $v^{l (w) / 2} C'_w = \sum_{u
\leqslant w} P_{u, w} (v) T_u$, where $P_{u, w} (v) \in \mathbb{Z} [v]$ of
degree $\leqslant \frac{1}{2} (l (w) - l (u) - 1)$ for $y < w$ and $P_{w, w} =
1$. The $P_{u, w}$ are the Kazhdan-Lusztig polynomials. (Note that $P$ is
used differently in other parts of our paper.) Let $\boldsymbol{C}_w (\lambda) =
v^{l (w) / 2} C'_w \cdot \boldsymbol{z}^{\lambda}$, where the action is that of
Theorem~\ref{cherednik}.

\begin{proposition}
  Let $\lambda$ be a dominant weight and $w \in W$. We have $\boldsymbol{X}_w
  (\lambda) =\boldsymbol{C}_w (\lambda)$ if and only if the Schubert variety
  $X_w$ is rationally smooth in the sense of Kazhdan and Lusztig~{\cite{KL}}.
\end{proposition}

\begin{proof}
  Each Kazhdan-Lusztig $P_{u, w}$ with $u \leqslant w$ is a polynomial in $v$
  with constant term equal to $1$, so $\boldsymbol{X}_w (\lambda)
  =\boldsymbol{C}_w (\lambda)$ if and only if $P_{u, w} = 1$ for all $u
  \leqslant w$. By Theorem~A2 of {\cite{KL}}, this is true if and only if
  $X_w$ is rationally smooth.
\end{proof}

\end{document}